\documentclass[12pt]{article}
\usepackage{a4wide}
\usepackage[french,english]{babel}
\usepackage[utf8]{inputenc}
\usepackage[T1]{fontenc}
\usepackage{amsthm}
\usepackage{amssymb}
\usepackage{amsmath}
\usepackage{stmaryrd}
\usepackage{makeidx}
\usepackage{dsfont}
\usepackage{graphicx}
\usepackage{caption}
\theoremstyle{plain}
\usepackage[colorlinks=true,breaklinks=true,linkcolor=black]{hyperref} %pour faire des liens hypertextes
\usepackage{amsfonts}
\usepackage{esvect}

\newtheorem{Theorem}{Theorem}[section] %[subsection sert à numéroter les théorèmes d'après la sous-section où ils se trouvent
\newtheorem{Proposition}[Theorem]{Proposition}
\newtheorem{Corollary}[Theorem]{Corollary}
\newtheorem{Definition}[Theorem]{Definition}
\newtheorem{Lemma}[Theorem]{Lemma}

\newtheorem{Remark}[Theorem]{Remark}

\newtheorem{Definition/proposition}[Theorem]{Definition/Proposition}
\theoremstyle{Definition}\newtheorem{Theorem*}{Theorem}
\theoremstyle{definition}\newtheorem{Corollary*}[Theorem*]{Corollary}
\theoremstyle{definition}\newtheorem{Proposition*}[Theorem*]{Proposition}
\date{}
\makeindex
\usepackage{mathrsfs}
\usepackage[all]{xy}
\theoremstyle{Definition}

\makeatletter
\renewcommand\theequation%
{\thesection.\arabic{equation}}
\@addtoreset{equation}{section}
\makeatother

\usepackage{geometry}
\title{Kato's irrreducibility criterion for Kac-Moody groups over local fields}
\author{Auguste \textsc{Hébert}}
\usepackage{mathrsfs}

\makeatletter \@addtoreset{figure}{section}\makeatother

\newcommand{\R}{\mathbb{R}}
\newcommand{\A}{\mathbb{A}}

\newcommand{\N}{\mathbb{N}}
\newcommand{\Z}{\mathbb{Z}}

\newcommand{\C}{\mathbb{C}}

\newcommand{\I}{\mathcal{I}}
\newcommand{\T}{\mathcal{T}}
\newcommand{\Id}{\mathrm{Id}}

\newcommand{\supp}{\mathrm{supp}}

\newcommand{\AC}{{^{\mathrm{BL}}\mathcal{H}}}
\newcommand{\ATF}{\AC(T_\C)}
\newcommand{\AF}{\AC_\C}
\newcommand{\ATC}{\AC(T_\C)}

\newcommand{\Hom}{\mathrm{Hom}}

\newcommand{\DC}{\mathcal{D}}
\newcommand{\FC}{\mathcal{F}}

\newcommand{\HC}{\mathcal{H}}
\newcommand{\KC}{\mathcal{K}}

\newcommand{\NC}{\mathcal{N}}

\newcommand{\RC}{\mathcal{R}}
\newcommand{\SC}{\mathcal{S}}
\newcommand{\UC}{\mathcal{U}}

\newcommand{\KCC}{\mathscr{K}}

\newcommand{\JC}{\mathcal{J}}
\newcommand{\HCW}{\mathcal{H}_{W^v,\C}}

\newcommand{\HFW}{{\mathcal{\HC}_{W^v,\FC}}}

\newcommand{\LCC}{\mathscr{L}}

\newcommand{\RCC}{\mathscr{R}}
\newcommand{\SCC}{\mathscr{S}}

\newcommand{\End}{\mathrm{End}}

\newcommand{\Wta}{W_{(\tau)}}
\newcommand{\vb}{\mathbf{v}}
\newcommand{\vbt}{\textbf{v}_{\tilde{\tau}}}
\newcommand{\tKC}{\tilde{\KC}}
\newcommand{\Itg}{I_\tau(\tau,\mathrm{gen},\Wta)}
\newcommand{\LTI}{\mathrm{LT}}
\newcommand{\ord}{\mathrm{ord}}
\newcommand{\htt}{\mathrm{ht}}

\usepackage{comment}

\newcommand{\ev}{\mathrm{ev}}

\makeindex

\usepackage[left,modulo]{lineno}

\hypersetup{
pdfauthor={Hebert, Auguste},
pdftitle={Kato's irreducibility criterion},
}

\begin{document}

%\linenumbers

\maketitle

\begin{abstract}
In 2014, Braverman, Kazhdan, Patnaik and Bardy-Panse, Gaussent and Rousseau associated Iwahori-Hecke algebras to Kac-Moody groups over non-Archimedean local fields. In a previous paper, we defined and studied their principal series representations.  In 1982, Kato provided an irreducibility criterion for these representations, in the reductive case. We had  obtained partially this criterion in the Kac-Moody case. In this paper, we prove this criterion in the Kac-Moody case.
\end{abstract}

\section{Introduction}

\subsection{The reductive case}

Let  $G$ be a split reductive group over a non-Archimedean local field $\KC$. 
 To each open compact subgroup $K$ of $G$ is associated the Hecke algebra $\HC_K$. This is the algebra of functions from $G$ to $\C$ which have compact support and are $K$-bi-invariant.  There exists a strong link between the smooth representations of $G$ and the representations of the Hecke algebras of $G$. Let $K_I$ be the Iwahori subgroup of $G$. Then the Hecke algebra $\HC_\C$ associated with $K_I$ is called the Iwahori-Hecke algebra of $G$ and plays an important role in the representation theory of $G$. 
 
 Let $T$ be a maximal split torus of $G$ and $Y$ be the cocharacter lattice of $(G,T)$. Let $B$ be a Borel subgroup of $G$ containing $T$.  Let $T_\C=\Hom_{\mathrm{Gr}}(Y,\C^*)$ and $\tau\in T_\C$. Then $\tau$ can be extended to a character $\tau:B\rightarrow \C^*$. If $\tau\in T_\C$,  the principal series representation $I(\tau)$  of $G$ is the induction of $\tau\delta^{1/2}$ from $B$ to $G$, where $\delta:B\rightarrow \R^*_+$ is the modulus character of $B$. More explicitly, this  is the space of locally constant functions $f:G\rightarrow \C$ such that $f(bg)=\tau\delta^{1/2}(b)f(g)$ for every $g\in G$ and $b\in B$. Then $G$ acts on $I(\tau)$ by right translation.

Let  $W^v$ be the vectorial Weyl group of $(G,T)$. By the Bernstein-Lusztig relations, $\HC_\C$ admits a basis  $(Z^\lambda* T_w)_{\lambda\in Y, w\in W^v}$ such that $\bigoplus_{\lambda\in Y}\C Z^\lambda$ is a subalgebra of $\HC_\C$ isomorphic to the group algebra $\C[Y]$ of $Y$. We identify $\bigoplus_{\lambda\in Y}\C Z^\lambda$ and $\C[Y]$. We regard $\tau$ as an algebra morphism $\tau:\C[Y]\rightarrow\C$. Then the algebra  $\HC_\C$ acts on $I_{\tau,G}:=I(\tau)^{K_I}$,  $I(\tau)$ is irreducible as a representation of $G$ if and only $I_{\tau,G}$ is irreducible as a representation of $\HC_\C$ and $I_{\tau,G}$ is isomorphic to the induced representation $I_\tau=\mathrm{Ind}_{\C[Y]}^{\HC_\C}(\tau)$.

Matsumoto and Kato gave criteria for the irreducibility of $I_\tau$.  The group  $W^v$ acts on $Y$ and thus it acts on $T_\C$. If $\tau\in T_\C$, we denote by $W_\tau$ the stabilizer  of $\tau$ in $W^v$. Denote by $q$ be the residue cardinal of $\KC$.  Let $\Wta$ be the subgroup of $W_\tau$ generated by the reflections $r_{\alpha^\vee}$, for $\alpha^\vee\in \Phi^\vee$ such that $\tau(\alpha^\vee)=1$, where $\Phi^\vee$ stands for the coroot lattice of $G$. Then Kato proved the following theorem (see \cite[Theorem 2.4]{kato1982irreducibility}):
 
 \begin{Theorem*}\label{thm*Kato's theorem}
 Let $\tau\in T_\C$. Then $I_\tau$ is irreducible if and only if it satisfies the following conditions: \begin{enumerate}
\item\label{itWchi engendré par ses réflexions} $W_\tau=\Wta$,

\item for all $\alpha^\vee\in \Phi^\vee$, $\tau(\alpha^\vee)\neq q$. 
 \end{enumerate}
 \end{Theorem*}

When $\tau$ is \textbf{regular}, that is when $W_\tau=\{1\}$, condition~(\ref{itWchi engendré par ses réflexions}) is satisfied and this is a result by Matsumoto (see \cite[Th{\'e}or{\`e}me 4.3.5]{matsumoto77Analyse}).

\subsection{The Kac-Moody case}

Let $G$ be a split Kac-Moody group over a non-Archimedean local field $\KC$. We do not know which topology on $G$ could replace the usual topology on reductive groups over $\KC$.  There is up to now no definition of smoothness for the representations of $G$. However one can define certain Hecke algebras in this framework. In \cite{braverman2011spherical} and \cite{braverman2016iwahori}, Braverman, Kazhdan and Patnaik defined the spherical Hecke algebra and the Iwahori-Hecke $\HC_\C$ of $G$ when $G$ is affine. In \cite{gaussent2014spherical} and \cite{bardy2016iwahori}, Bardy-Panse, Gaussent and Rousseau generalized these constructions to the case where $G$ is a general Kac-Moody group. They achieved this construction by using masures (also known as hovels), which are analogous to Bruhat-Tits buildings (see \cite{gaussent2008kac}).

Let $B$ be a positive Borel subgroup of $G$ and $T$ be a maximal split torus of $G$ contained in $B$.  Let $Y$ be the cocharacter lattice of $G$, $W^v$ be the Weyl group of $G$ and $Y^{++}$ be the set of dominant cocharacters of $Y$. The Bruhat decomposition does not hold on $G$: if $G$ is not reductive, \[G^+:=\bigsqcup_{\lambda\in Y^{++} }K_I \lambda K_I\subsetneq G.\]  The set $G^+$ is a sub-semi-group of $G$. Then $\HC_\C$ is defined to be the set of functions from $K_I\backslash G^+/K_I$ to $\C$ which have finite support.  The Iwahori-Hecke algebra $\HC_\C$ of $G$ admits a Bernstein-Lusztig presentation but it is no longer indexed by $Y$. Let $Y^+=\bigcup_{w\in W^v} w.Y^{++}\subset Y$. Then $Y^+$ is the \textbf{integral Tits cone} and we have $Y^+=Y$ if and only $G$ is reductive. The \textbf{Bernstein-Lusztig-Hecke algebra of }$G$ is the space $\AC_\C=\bigoplus_{w\in W^v} \C[Y] *T_w$ subject to to some relations (see subsection~\ref{subIH algebras}). Then $\HC_\C$ is isomorphic to $\bigoplus_{w\in W^v} \C[Y^+]* T_w$. 

Let $\tau\in T_\C=\Hom_{\mathrm{Gr}}(Y,\C^*)$. In  \cite[6]{hebert2018principal} we defined the space $\widehat{I(\tau)}$ of functions $f$ from $G$ to $\C$ such that for all $g\in G$, $b\in B$, we have $f(bg)=\tau\delta^{1/2}(b)f(g)$. As we do not know which condition could replace 
``locally constant'', we do not impose any regularity condition on the functions of $\widehat{I(\tau)}$. Then $G$  acts by right translation on $\widehat{I(\tau)}$.
Let $I_{\tau,G}$ be the subspace of functions $f\in \widehat{I(\tau)}$   which are invariant under the action of $K_I$ and whose support satisfy some finiteness conditions. We defined an action of $\HC_\C$ on $I_{\tau,G}$. This actions extends to an action of $\AC_\C$ on $I_{\tau,G}$ and is isomorphic to the induced representation $I_\tau=\mathrm{Ind}_{\C[Y]}^{\AC_\C}(\tau)$. Moreover the $\AC_\C$-submodules of $I_\tau$ are exactly the $\HC_\C$-submodules of $I_\tau$ (see \cite[Proposition 3.1]{hebert2021decompositions}) and thus we regard $I_\tau$ as a $\AC_\C$-module. We then obtained a weak version of Theorem~\ref{thm*Kato's theorem}: we obtained one implication and we proved the equivalence only under the assumption that the Kac-Moody matrix defining $G$ has size $2$ (see \cite[Theorem 3 and 4]{hebert2018principal}). In this paper, we prove Theorem~\ref{thm*Kato's theorem} in a full generality (see Corollary~\ref{corKatos_irreducibility_criterion}).

\paragraph{Basic ideas of the proof} Let us explain the basic ideas of our proof. We have $I_\tau=\bigoplus_{w\in W^v} \C T_w .\vb_\tau$, where $\vb_\tau\in I_\tau$ is such that $\theta.\vb_\tau=\tau(\theta).\vb_\tau$ for $\theta\in \C[Y]$. For $w\in W^v$, one sets \[I_\tau(w.\tau)=\{x\in I_\tau|\theta.x=w.\tau(\theta).x,\ \forall\theta\in \C[Y]\}.\] By the  Frobenius reciprocity, $\Hom_{\AC_\C-\mathrm{mod}}(I_\tau,I_{w.\tau})$ is isomorphic to $I_\tau(w.\tau)$ as a vector space. Let $\UC_\C$ be the set of $\tau\in T_\C$ such that $\tau(\alpha^\vee)\neq q$ for all $\alpha^\vee\in \Phi^\vee$. Let $\tau\in \UC_\C$. In \cite[Theorem 4.8]{hebert2018principal} we proved that $I_\tau$ is irreducible if and only if $\End_{\AC_\C-\mathrm{mod}}(I_\tau)=\C.\Id$, if and only if $I_\tau(\tau)=\C\vb_\tau$. 

In order to study $I_\tau(w.\tau)$, for $w\in W^v$, it is convenient to introduce $\AC(T_\C)=\bigoplus_{w\in W^v} T_w *\C(Y)$. The elements of $\AC(T_\C)$ can be regarded as rational functions from an open subset of $T_\C$ to $\HC_{W^v,\C}=\bigoplus_{w\in W^v} \C T_w$. Following Reeder, we introduced elements $F_w\in \ATC$, $w\in W^v$, such that for all $\chi\in T_\C$ for which $F_w(\chi)$ is well defined, $F_w(\chi).\vb_\chi\in I_\chi(w.\chi)$. The group $W_\tau$ decomposes as $W_\tau=R_\tau\ltimes \Wta$,  where $R_\tau$ is some subgroup of $W_\tau$ called the $R$-group. If $w_R\in R_\tau$, then $F_{w_R}$ has no pole at $\tau$ and thus $F_{w_R}(\tau).\vb_\tau$ corresponds to an element $\psi_{w_r}$ of $\End(I_\tau)$. For $w\in \Wta$ however, $F_w$ has poles at $\tau$ and thus describing $I_\tau(\tau)$ requires some works.   Inspired by works of Reeder and Keys in the reductive case, (\cite{reeder1997nonstandard} and \cite{keys1982decomposition}), we determined \[I_\tau(\tau,\mathrm{gen})=\{x\in I_\tau|\forall \theta\in \C[Y], \forall n\gg 0, \big(\theta-\tau(\theta)\big)^n.x=0\}.\] We proved (see \cite[Proposition 5.13]{hebert2021decompositions}) that \[I_\tau(\tau,\mathrm{gen})=\bigoplus_{w_R\in R_\tau} \psi_{w_R}\left(\Itg\right),\] where $\Itg:=\left(\AC_\C\cap  \bigoplus_{w\in \Wta}F_w*\C(Y)\right).\vb_\tau$ corresponds to the ``$\Wta$-part'' of $I_\tau(\tau,\mathrm{gen})$. We deduced that \[I_\tau(\tau)=\bigoplus_{w_R\in R_\tau} \psi_{w_R}\left(I_\tau(\tau)\cap \Itg\right).\]

It then remained to prove that $\Itg\cap I_\tau(\tau)=\C \vb_\tau$, which we achieve in this paper (see Theorem~\ref{thmWeight_space}). Note that by our description of  $\Itg$, proving that $I_\tau(\tau)\cap \Itg=\C\vb_\tau$ can more or less be reduced to proving that  $I_{\mathds{1}}(\mathds{1})=\C \vb_{\mathds{1}}$, where $I_{\mathds{1}}=\mathrm{Ind}_{\C(Y)_\tau}^{\KCC_\tau}(\mathds{1})$,    $\C(Y)_\tau$ is the subset of $\C(Y)$ consisting  of the elements which have no pole at $\tau$, $\mathds{1}:Y\rightarrow \C$ is the constant function equal to $1$ and $\KCC_\tau\subset \bigoplus_{w\in W^v} T_w*\C(Y)_\tau$ is some kind of Bernstein-Lusztig-Hecke algebra associated with $\tau$ (see subsection~\ref{subBLH_structure_K} for the definition of $\KCC_\tau$).  

We then deduce:

\begin{Corollary*}\label{cor*End_I_tau}(see Corollary~\ref{corSilberger_dimension_theorem} and Corollary~\ref{corKatos_irreducibility_criterion})
Let $\tau\in \UC_\C$. Then $\End(I_\tau)\simeq \C[R_\tau]$, where $R_\tau=W_\tau/\Wta$. In particular, $\End(I_\tau)=\C\Id$ if and only $R_\tau=\{1\}$.
\end{Corollary*}

In \cite{hebert2021decompositions}, we studied the submodules and the quotients of $I_\tau$, for $\tau\in \UC_\C$. Many results were proved  only when the Kac-Moody matrix defining $G$ has size $2$ or under some conjecture (\cite[Conjecture 5.16]{hebert2021decompositions}). As we prove this conjecture, we can drop the assumption on the size of the matrix. In particular, \cite[Theorem 5.34 and Theorem 5.38]{hebert2021decompositions} yield links between the submodules and the quotients of $I_\tau$  and the right submodules and quotients of $\End(I_\tau)$ respectively.

\paragraph{Frameworks}
Actually, following \cite{bardy2016iwahori} we study Iwahori-Hecke algebras associated with abstract masures. In particular our results also apply when $G$ is an almost-split Kac-Moody group over a non-Archimedean local field. In this case, most of the results of this introduction are true but the formulas are more complicated (they are given in the paper). Corollary~\ref{cor*End_I_tau} is not necessarily true for almost-split groups, even in the reductive case.

\paragraph{Organization of the paper}
This paper is organized as follows. In section~\ref{secIH algebras}, we recall the definitions of the Iwahori–Hecke algebras and of the principal series representations, and we introduce tools to study these representations.

In section~\ref{secDescription_Itg}, we study the algebra $\KCC_\tau$ mentioned above and describe $\Itg$.

In section~\ref{secKato_s_irreducibility_criterion}, we prove Kato's irreducibility criterion.

\section{Iwahori-Hecke algebras}\label{secIH algebras}\label{secIH_algebras}
Let $G$ be a Kac-Moody group over a non-archimedean local field. Then Gaussent and Rousseau constructed a space $\I$, called a masure on which $G$ acts, generalizing the construction of the Bruhat-Tits buildings (see \cite{gaussent2008kac}, \cite{rousseau2016groupes} and \cite{rousseau2017almost}). In \cite{bardy2016iwahori} Bardy-Panse, Gaussent and Rousseau attached an Iwahori-Hecke algebra $\HC_\RC$ to each  masure satisfying certain conditions and to each ring $\RC$. They in particular attach an Iwahori-Hecke algebra to each almost-split Kac-Moody group over a local field. The algebra $\HC_\RC$ is an algebra of functions defined on some pairs of chambers of the masure, equipped with a convolution product. Then they prove that under some additional hypothesis on the ring $\RC$ (which are satisfied by  $\C$),  $\HC_\RC$  admits a Bernstein-Lusztig presentation.  We restrict our study to the case where $\RC=\C$.  In this paper, we will only use the Bernstein-Lusztig presentation of $\HC_\C$. More precisely, we introduce an algebra $\AC(T_\C)=\bigoplus_{w\in W^v} T_w*\C(Y)$, which contains both $\AC_\C$ and $\HC_\C$. We mainly study $\AC(T_\C)$ and $\AC_\C$. We do not introduce masures nor Kac-Moody groups. We however introduce the standard apartment of a masure.

\subsection{Standard apartment of a masure}\label{subRootGenSyst}
A (finite)\textbf{ Kac-Moody matrix} (or { generalized Cartan matrix}) is a square matrix $A=(a_{i,j})_{i,j\in I}$ indexed by a finite set $I$, with integral coefficients, and such that :
\begin{enumerate}
\item[\tt $(i)$] $\forall \ i\in I,\ a_{i,i}=2$;

\item[\tt $(ii)$] $\forall \ (i,j)\in I^2, (i \neq j) \Rightarrow (a_{i,j}\leq 0)$;

\item[\tt $(iii)$] $\forall \ (i,j)\in I^2,\ (a_{i,j}=0) \Leftrightarrow (a_{j,i}=0$).
\end{enumerate}

In the paper, we will also consider infinite Kac-Moody matrices. The definition is the same except that $I$ is infinite. However we will only consider root generating systems associated with finite Kac-Moody matrices.

A \textbf{root generating system} is a $5$-tuple $\mathcal{S}=(A,X,Y,(\alpha_i)_{i\in I},(\alpha_i^\vee)_{i\in I})$\index{$\mathcal{S}$} made of a finite Kac-Moody matrix $A$ indexed by the finite set $I$, of two dual free $\Z$-modules $X$ and $Y$ of finite rank, and of a free family $(\alpha_i)_{i\in I}$ (respectively $(\alpha_i^\vee)_{i\in I}$) of elements in $X$ (resp. $Y$) called \textbf{simple roots} (resp. \textbf{simple coroots}) that satisfy $a_{i,j}=\alpha_j(\alpha_i^\vee)$ for all $i,j$ in $I$. Elements of $X$ (respectively of $Y$) are called \textbf{characters} (resp. \textbf{cocharacters}).

Fix such a root generating system $\mathcal{S}=(A,X,Y,(\alpha_i)_{i\in I},(\alpha_i^\vee)_{i\in I})$ and set $\A:=Y\otimes \R$\index{$\A$}. Each element of $X$ induces a linear form on $\A$, hence $X$ can be seen as a subset of the dual $\A^*$. In particular, the $\alpha_{i}$'s (with $i \in I$) will be seen as linear forms on $\A$. This allows us to define, for any $i \in I$, an involution $r_{i}$ of $\A$ by setting $r_{i}(v) := v-\alpha_i(v)\alpha_i^\vee$ for any $v \in \A$. Let $\SCC=\{r_i|i\in I\}$\index{$\SCC$} be the (finite) set of \textbf{simple reflections}.  One defines the \textbf{Weyl group of $\mathcal{S}$} as the subgroup $W^{v}$\index{$W^v$} of $\mathrm{GL}(\A)$ generated by $\SCC$. The pair $(W^{v}, \SCC)$ is a Coxeter system, hence we can consider the length $\ell(w)$ with respect to $\SCC$ of any element $w$ of $W^{v}$. If $s\in \SCC$, $s=r_i$ for some unique $i\in I$. We set $\alpha_s=\alpha_i$ and $\alpha_s^\vee=\alpha_i^\vee$.

The following formula defines an action of the Weyl group $W^{v}$ on $\A^{*}$:  
\[\displaystyle \forall \ x \in \A , w \in W^{v} , \alpha \in \A^{*} , \ (w.\alpha)(x):= \alpha(w^{-1}.x).\]
Let $\Phi:= \{w.\alpha_i|(w,i)\in W^{v}\times I\}$\index{$\Phi,\Phi^\vee$} (resp. $\Phi^\vee=\{w.\alpha_i^\vee|(w,i)\in W^{v}\times I\}$) be the set of \textbf{real roots} (resp. \textbf{real coroots}): then $\Phi$ (resp. $\Phi^\vee$) is a subset of the \textbf{root lattice} $Q_\Z:= \displaystyle \bigoplus_{i\in I}\Z\alpha_i$ (resp. \textbf{coroot lattice} $Q^\vee_\Z=\bigoplus_{i\in I}\Z\alpha_i^\vee$). By \cite[1.2.2 (2)]{kumar2002kac}, we have $\R \alpha^\vee\cap \Phi^\vee=\{\pm \alpha^\vee\}$ and $\R \alpha\cap \Phi=\{\pm \alpha\}$ for all $\alpha^\vee\in \Phi^\vee$ and $\alpha\in \Phi$.

\paragraph{Reflections and roots}

We equip $(W^v,\SCC)$ with the Bruhat order $\leq$ (see \cite[Definition 2.1.1]{bjorner2005combinatorics}).

Let $\RCC=\{wsw^{-1}|w\in W^v, s\in \SCC\}$\index{$\RCC$} be the set of \textbf{reflections} of $W^v$. Let  $r\in \RCC$. Write $r=wsw^{-1}$, where $w\in W^v$, $s\in \SCC$ and $ws>w$ (which is possible because if $ws<w$, then $r=(ws)s(ws)^{-1}$). Then one sets $\alpha_r=w.\alpha_s\in \Phi_+$\index{$\alpha_r,\alpha_r^\vee$} (resp. $\alpha_r^\vee=w.\alpha_s^\vee\in\Phi^\vee_+$). Conversely, if $\alpha\in \Phi$ or $\alpha^\vee\in \Phi^\vee$, $\alpha=w.\alpha_s$ or $\alpha^\vee=w.\alpha_s^\vee$, where $w\in W^v$ and $s\in \SCC$, one sets $r_{\alpha}=wsw^{-1}$ or $r_{\alpha^\vee}=wsw^{-1}$.  This is independant of the choices of $w$ and $s$ by \cite[1.3.11 Theorem (b5)]{kumar2002kac}.

\subsection{Iwahori-Hecke algebras}\label{subIH algebras}
In this subsection, we give the definition of  the Iwahori-Hecke algebra via its Bernstein-Lusztig presentation.
\subsubsection{The algebra $\AC(T_\C)$}\label{subsubAlgebra_H(T_F)}

Let $(\sigma_s)_{s\in \SCC}, (\sigma'_s)_{s\in \SCC}\in \C^{\SCC}$ satisfying the following relations: \begin{itemize}
\item if $\alpha_{s}(Y) = \Z$, then $\sigma_{s} = \sigma'_{s}$\index{$\sigma_s,\sigma_s'$};  
\item if $s,t \in \SCC$ are conjugate (i.e if there exists a sequence $s_1,\ldots,s_n\in \SCC$ such that $s_1=s$, $s_n=t$ and $\alpha_{s_i}(\alpha_{s_{i+1}}^\vee)=\alpha_{s_{i+1}}(\alpha_{s_i}^\vee)=-1$, for $i\in \llbracket 1,n-1\rrbracket$), then $\sigma_{s}=\sigma_{t}=\sigma'_{s}=\sigma'_{t}$.
\end{itemize} As in \cite{hebert2018principal}, we moreover assume that $|\sigma_s|,|\sigma_s'|>1$, for all $s\in \SCC$.

\begin{Definition} Let $\HC_{W^v,\C}$ be the \textbf{Hecke algebra of the Coxeter group $W^v$ over $\C$}, that is: \begin{itemize}

\item  as a vector space, $\HC_{W^v,\C}=\bigoplus_{w\in W^v} \C T_w$, where the $T_w$, $w\in W^v$ are symbols,

\item $\forall \ s \in \SCC, \forall \ w \in W^{v}$, $T_{s}*T_{w}=\left\{\begin{aligned} & T_{sw} &\mathrm{\ if\ }\ell(sw)=\ell(w)+1\\ & (\sigma_{s}^2-1) T_{w}+\sigma_s^2 T_{s w} &\mathrm{\ if\ }\ell(sw)=\ell(w)-1 .\end{aligned}\right . \ $
\end{itemize}

\end{Definition}

Let $\C[Y]$ be the group algebra of $Y$ over $\C$, that is:\begin{itemize}
\item as a vector space, $\C[Y]=\bigoplus_{\lambda\in Y} \C Z^\lambda$, where the $Z^\lambda$, $\lambda\in Y$ are symbols,

\item for all $\lambda,\mu\in Y$, $Z^\lambda*Z^\mu=Z^{\lambda+\mu}$.
\end{itemize}

 We denote by $\C(Y)$ its field of fractions. For $\theta=\frac{\sum_{\lambda\in Y}{a_\lambda Z^\lambda}}{\sum_{\lambda\in Y} b_\lambda Z^\lambda}\in \C(Y)$ and $w\in W^v$, set ${^w}\theta:=\frac{\sum_{\lambda\in Y}a_\lambda Z^{w.\lambda}}{\sum_{\lambda\in Y}b_\lambda Z^{w.\lambda}}$. 

\medskip

Let $\AC(T_\C)$ be the algebra defined as follows: \begin{itemize}
\item as a vector space, $\AC(T_\C)= \C(Y)\otimes \HC_{W^v, \C}$ (we write $\theta*h$ instead of $\theta\otimes h$ for $\theta\in \C(Y)$ and $h\in \HC_{W^v,\C}$),

\item $\AC(T_\C)$ is equipped with the unique product $*$ which turns it into an associative algebra and such that, for  $\theta\in \C(Y)$ and $s\in \SCC$, we have: \[T_{s}*\theta-{^s\theta}*T_{s} =\tilde{\Omega}_s(\theta-{^s\theta}),\] where $\tilde{\Omega}_s(\theta)=Q_s^T(\theta-{^s\theta})$ and $Q_s^T=\frac{(\sigma_s^2-1)+\sigma_s(\sigma_s'-\sigma_s'^{-1})Z^{-\alpha_s^\vee}}{1-Z^{-2\alpha_s^\vee}}$\index{$Q_s^T$}. 

\end{itemize}

By \cite[Proposition 2.10]{hebert2018principal}, such an algebra exists and is unique.

\subsubsection{The Bernstein-Lusztig Hecke algebra and the Iwahori-Hecke algebra}

Let $C^v_f=\{x\in \A|\alpha_i(x)>0\forall i\in I\}$\index{$C^v_f$}, $\T=\bigcup_{w\in W^v} w.\overline{C}^v_f$\index{$\T$} be the \textbf{Tits cone} and $Y^+=Y\cap \T$\index{$Y^+$}.

\begin{Definition}\label{defBernstein-Lusztig algebra}
 The \textbf{Bernstein-Lusztig-Hecke algebra of } $(\A,(\sigma_s)_{s\in \SCC},(\sigma'_s)_{s\in \SCC})$ over $\C$ is the subalgebra $\AC_\C=\bigoplus_{\lambda\in Y,w\in W^v}\C Z^\lambda*T_w=\bigoplus_{\lambda\in Y,w\in W^v}\C  T_w* Z^\lambda$ of $\AC(T_\C)$. The \textbf{Iwahori-Hecke algebra of $(\A,(\sigma_s)_{s\in \SCC},(\sigma'_s)_{s\in \SCC})$ over $\C$} is the subalgebra $\HC_\C=\bigoplus_{\lambda\in Y^+,w\in W^v}\C Z^\lambda*T_w=\bigoplus_{\lambda\in Y^+,w\in W^v}\C T_w*Z^\lambda$ of $\AC_\C$. Note that for $G$ reductive, we recover the usual Iwahori-Hecke algebra of $G$, since $ \T=\A$.
\end{Definition}

\begin{Remark}\label{remIH algebre dans le cas KM deploye}

\begin{enumerate}
\item The algebra $\AC_\C$  was first defined in \cite[Theorem 6.2]{bardy2016iwahori} without defining $\AC(T_\C)$. Let $\mathcal{K}$ be a non-Archimedean local field and $q$ be its residue cardinal. Let $\mathbf{G}$ be the minimal Kac-Moody group associated with $\mathcal{S}=(A,X,Y,(\alpha_i)_{i\in I},(\alpha_i^\vee)_{i\in I})$\index{$\mathcal{S}$} and $G=\mathbf{G}(\mathcal{K})$ (see \cite[Section 8]{remy2002groupes} or \cite{tits1987uniqueness} for the definition). Take $\sigma_s=\sigma'_s=\sqrt{q}$ for all $s\in \SCC$. Then $\HC_\C$ is the Iwahori-Hecke algebra of $G$ (see \cite[Definition 2.5 and 6.6 Proposition]{bardy2016iwahori}). In the case where $G$ is an untwisted affine Kac-Moody group, these algebras were introduced in \cite{braverman2016iwahori}. Note also  that our frameworks is more general than the one of split Kac-Moody groups over local fields. It enables for example to study the Iwahori-Hecke algebras associated with almost split Kac-Moody groups over local fields, as in \cite{bardy2016iwahori}. In this case we do not have necessarily $\sigma_s=\sigma_s'=\sigma_t=\sigma_t'$ for all $s,t\in \SCC$.

\item Let $s\in \SCC$. Then if $\sigma_s=\sigma'_s$, $Q_s^T=\frac{\sigma_s^2-1}{1-Z^{-\alpha_s^\vee}}$.

\item\label{itPolynomiality_Bernstein_Lusztig} Let $s\in \SCC$. Then $\tilde{\Omega}_s(\C[Y])\subset \C[Y]$ and $\tilde{\Omega}_s(\C[Y^+])\subset \C[Y^+]$ Indeed, let $\lambda\in Y$. Then $Q_s^T(Z^\lambda-Z^{s.\lambda})=Q^T_s.Z^\lambda(1-Z^{-\alpha_s(\lambda)\alpha_s^\vee})$. Assume that $\sigma_s=\sigma_s'$. Then \[ \frac{1-Z^{-\alpha_s(\lambda)\alpha_s^\vee}}{1-Z^{-\alpha_s^\vee}}=\left\{\begin{aligned} &\sum_{j=0}^{\alpha_s(\lambda)-1}Z^{-j\alpha_s^\vee}  &\mathrm{\ if\ }\alpha_s(\lambda)\geq 0\\ & -Z^{\alpha_s^\vee}\sum_{j=0}^{-\alpha_s(\lambda)-1}Z^{j\alpha_s^\vee}  &\mathrm{\ if\ }\alpha_s(\lambda)\leq 0,\end{aligned}\right.\] and thus $Q_s^T(Z^\lambda-Z^{s.\lambda})\in \C[Y]$. If $\sigma_s'\neq \sigma_s$, then $\alpha_s(Y)=2\Z$ and a similar computation enables to conclude. In particular, $\AC_\C$ and $\HC_\C$ are subalgebras of $\AC(T_\C)$.

\item In \cite{hebert2018principal} and \cite{hebert2021decompositions}, we used the basis $(H_w)_{w\in W^v}$ instead of $(T_{w})_{w\in W^v}$ in the presentation of $\HCW$ and $\AC_\C$. If $w\in W^v$, $w=s_1\ldots s_k$, with $k=\ell(w)$ and $s_1,\ldots,s_k\in \SCC$, we have  $T_w=\sigma_{s_1}\ldots \sigma_{s_k} H_w$. 

\end{enumerate}

\end{Remark}

\subsection{Principal series representations}\label{subPrincipal series representations}
In this subsection, we introduce the principal series representations of $\AF$. 

We now fix $(\A,(\sigma_s)_{s\in \SCC},(\sigma'_s)_{s\in \SCC})$ as in Subsection~\ref{subIH algebras}. Let  $\AC_\C$ be the Iwahori-Hecke and the Bernstein-Lusztig Hecke algebras of $(\A,(\sigma_s)_{s\in \SCC},(\sigma'_s)_{s\in \SCC})$  over $\C$.

Let  $T_\C= \Hom_{\mathrm{Gr}}(Y,\C^*)$\index{$T_\C$} be the group of group morphisms from $Y$ to $\C^*$. Let $\tau\in T_\C$. Then $\tau$ induces an algebra morphism $\tau:\C[Y]\rightarrow \C$ by the formula $\tau(\sum_{\lambda\in Y} a_\lambda  Z^\lambda)=\sum_{\lambda\in Y} a_\lambda \tau(\lambda)$, for $\sum a_\lambda Z^\lambda\in \C[Y]$. This equips $\C$ with the structure of a  $\C[Y]$-module.

Let $I_\tau=\mathrm{Ind}^{\AC_\C}_{\C[Y]}(\tau)=\AC_\C\otimes_{\C[Y]} \C$\index{$I_\tau$}. As a vector space, $I_\tau=\bigoplus_{w\in W^v} \C \vb_\tau$, where $\vb_\tau$ is some symbol. The actions of $\AC_\C$ on $I_\tau$ is as follows. Let $h=\sum_{w\in W^v} T_w P_w\in \AC_\C$, where $P_w\in \C[Y]$ for all $w\in W^v$. Then $h.\vb_\tau=\sum_{w\in W^v} \tau(P_w)T_w\vb_\tau$. In particular, $I_\tau$ is a principal $\AC_\C$-module generated by $\vb_\tau$.

If $A$ is a vector space over $\C$ and $B$ is a set, we denote by $A^{(B)}$ the set of families $(a_b)_{b\in B}$ such that $\{b\in B|a_b\neq 0\}$ is finite. We regard the elements of $\C[Y]$ as polynomial functions on $T_\C$ by setting: \[\tau(\sum_{\lambda\in Y}a_\lambda Z^\lambda)=\sum_{\lambda\in Y}a_\lambda\tau(\lambda),\] for all $(a_\lambda)\in \C^{(Y)}$. The ring $\C[Y]$ is a unique factorization domain. Let $\theta\in \C(Y)$ and $(f,g)\in \C[Y]\times \C[Y]^*$ be such that $\theta=\frac{f}{g}$ and $f$ and $g$ are coprime. Set $\DC(\theta)=\{\tau\in T_\C|\tau(g)\neq 0\}$. Then we regard  $\theta$ as a map from $\DC(\theta)$ to $\C$ by setting $\theta(\tau)=\frac{f(\tau)}{g(\tau)}$ for all $\tau\in \DC(\theta)$.

If $\tau\in T_\C$, let $\C(Y)_\tau=\{\frac{f}{g}|f,g\in \C[Y]\text{ and } g(\tau)\neq 0\}\subset \C(Y)$\index{$\C(Y)_\tau$}. Let $\ATF_\tau=\bigoplus_{w\in W^v} T_w \C(Y)_\tau \subset \ATF$\index{$\ATF_\tau,\ATC_\tau$}. This is a not a subalgebra of $\ATF$ (consider for example $\frac{1}{Z^\lambda-1}*T_s=T_s*\frac{1}{Z^{s.\lambda}-1}+\ldots$ for some well chosen $\lambda\in Y$, $s\in \SCC$ and $\tau\in T_\C$). It is however an $\HFW-\C(Y)_\tau$ bimodule. For $\tau\in T_\C$, we define $\ev_\tau:\ATF_\tau\rightarrow \HFW$\index{$ev_\tau$} by $\ev_\tau(h)=h(\tau)=\sum_{w\in W^v} T_w\theta_w(\tau)$ if $h=\sum_{w\in W^v}T_w \theta_w\in \ATC_\tau$. This is a morphism of $\HFW-\C(Y)_\tau$-bimodules.

\subsection{Decomposition of $W_\tau$}

For $r=wsw^{-1}\in \RCC$, with $w\in W^v$ and $s\in \SCC$, we define $Q^T_r={^w\left(Q_s^T\right)}$ and $\sigma_r=\sigma_s$. This does not depend on the choices of $w$ and $s$. For $r\in \RCC$, set  $\zeta_r=-Q^T_r+\sigma_r^2$. Write $\zeta_r=\frac{\zeta_r^{\mathrm{num}}}{\zeta_r^{\mathrm{den}}}$\index{$\zeta_r$}, with $\zeta_r^{\mathrm{num}},\zeta_{r}^{\mathrm{den}}$ are coprime elements of $\C[Y]$. 

 For $\tau\in T_\C$, set $W_\tau=\{w\in W^v|\ w.\tau=\tau\}$\index{$W_\tau$}, $\Phi^\vee_{(\tau)}=\{\alpha^\vee\in \Phi^\vee| \zeta_{\alpha^\vee}^{\mathrm{den}}(\tau)=0\}$\index{$\Phi^\vee_{(\tau)}$}, $\Phi^\vee_{(\tau),+}=\Phi^\vee_{(\tau)}\cap \Phi^\vee_+$,  $\RCC_{(\tau)}=\{r=r_{\alpha^\vee}\in \RCC|\alpha^\vee\in \Phi^\vee_{(\tau)}\}$\index{$\RCC_{(\tau)}$} and \[\Wta=\langle \RCC_{(\tau)}\rangle=\langle \{r=r_{\alpha^\vee}\in \RCC| \zeta_{\alpha^\vee}^{\mathrm{den}}(\tau)=0\}\rangle\subset W^v.\]\index{$\Wta$} 

 By \cite[Remark 5.1]{hebert2018principal}, $\Wta\subset W_\tau$. When $\alpha_s(Y)=\Z$ for all $s\in \SCC$, then $\Wta=\langle W_\tau\cap \RCC\rangle$. Let \[R_\tau=\{w\in W_\tau|w.\Phi^\vee_{(\tau),+}=\Phi^\vee_{(\tau),+}\}.\] By \cite[Lemma 5.2 and 5.3]{hebert2021decompositions}, $\Wta$ is normal in $W_\tau$ and $w_\tau=R_\tau\ltimes \Wta$. We defined in \cite[5.4]{hebert2018principal} a set $\SCC_\tau\subset\RCC\cap \Wta$ for which $(\Wta,\SCC_\tau)$ is a Coxeter system. We denote by $\leq_\tau$ the Bruhat order and by $\ell_\tau$ the length on $(\Wta,\SCC_\tau)$. Then by \cite[Lemma 5.10]{hebert2018principal}, for all $w,w'\in \Wta$, we have $w\leq_\tau w'$ implies $w\leq w'$.

\subsection{Weights and intertwining operators}

If $\tau\in T_\C$, we denote by $\End(I_\tau)$ the algebra of endomorphisms of $\AC_\C$-modules of $I_\tau$. Let  \[I_\tau(\tau)=\{x\in I_\tau|\theta.x=\tau(\theta).x\forall \theta\in \C[Y]\}\] and \[I_\tau(\tau,\mathrm{gen})=\{x\in I_\tau|\forall\theta\in \C[Y], (\theta-\tau(\theta))^n.x=0,\ \forall n\gg 0\}\supset I_\tau(\tau).\]

 For $s\in \SCC$, one sets \[F_s=T_s+Q_s^T\in \AC(T_\C)\index{$F_s$}.\]

Let $x\in I_\tau(\tau)$. Define $\Upsilon_x\in \End(I_\tau)$ by $\Upsilon(h.\vb_\tau)=h.x$, for $h\in \AC_\C$. Then it is easy to check that $\Upsilon:I_\tau(\tau)\rightarrow \End(I_\tau)$ is well defined and is an isomorphism of vector spaces.

 Let $w\in W^v$. Let $w=s_1\ldots s_r$ be a reduced expression of $w$, with $k=\ell(w)$ and $s_1,\ldots,s_k\in \SCC$. Set \[F_w=F_{s_r}\ldots F_{s_1}=(T_{s_r}+Q_{s_r}^T)\ldots (T_{s_1}+Q_{s_1}^T)\in \AC(T_\C).\]\index{$F_w$} 

\begin{Lemma}\label{lemReeder 4.3}(see \cite[Lemma 4.14]{hebert2018principal}) Let $w\in W^v$. \begin{enumerate}

\item\label{itWell_definedness_Fw} The element $F_w\in \AC(T_\C)$ is well defined, i.e it does not depend on the choice of a reduced expression for $w$.

\item\label{itLeading_coefficient_Fw} $F_w-T_w\in \ATF^{<w}=\bigoplus_{v<w}T_v\C(Y)$.

\item\label{itCommutation_relation} If $\theta\in \C(Y)$, then $\theta*F_w=F_w*{^{w^{-1}}}\theta$.

\item\label{itDomain_Fw} If $\tau\in T_\C$ is such that   $F_w\in \AC(T_\C)_\tau$, then $\theta.F_w(\tau).\vb_\tau =\left(w.\tau(\theta)\right) F_w(\tau).\vb_\tau$, for every $\theta\in \C[Y]$.

\end{enumerate}
\end{Lemma}

We set \[\UC_\C=\{\tau\in T_\C|\tau(\zeta_{r}^{\mathrm{num}})\neq 0\forall r\in \RCC\}.\] When $\sigma_s=\sigma_s'=\sqrt{q}$, for $q\in \R_{>0}$, $\UC_\C=\{\tau\in T_\C|\tau(\alpha^\vee)\neq q,\forall \alpha^\vee\in \Phi^\vee\}$.

\medskip

  By \cite[Lemma 5.7]{hebert2021decompositions}, if $w_R\in R_\tau$, then $F_{w_R}\in \ATC_\tau$ and $F_{w_R}(\tau).\vb_\tau\in I_\tau(\tau)$. Let $\psi_{w_R}=\Upsilon _{F_{w_R}(\tau).\vb_\tau}\in \End(I_\tau)$.

 Set \begin{equation}\label{eqDefinition_Itg}
\Itg= \left(\AC_\C \cap \bigoplus_{w\in \Wta} F_w *\C(Y)\right).\vb_\tau.
\end{equation} which is well defined by \cite[Lemma 5.23]{hebert2018principal}. \[  \sum_{k\in \N,s_1,\ldots, s_k\in \SCC_\tau} \C\ev_\tau(\KC_{s_1}*\ldots*\KC_{s_k}).\vb_\tau,\] By  \cite[Proposition 5.13 (1)]{hebert2021decompositions}, if $\tau\in \UC_\C$, we have   \begin{equation}\label{eqHeb21_5.12}
I_\tau(\tau,\mathrm{gen})=\bigoplus_{w_R\in R_\tau} \psi_{w_R}\big(\Itg\big).
 \end{equation}

\section{Description of  $\Itg$}\label{secDescription_Itg}

In this paper, we prove that if $\tau\in \UC_\C$,  then $I_\tau(\tau,\Wta):=\Itg\cap I_\tau(\tau)=\C \vb_\tau$ (see Theorem~\ref{thmWeight_space}) and we deduce information on the submodules of $I_\tau$. To that end,
we begin by describing $\Itg$.  For $s\in \SCC_\tau$, let \begin{equation}\label{eqDef_Ks}
\KC_s=F_s-\zeta_s=F_s+Q^T_s-\sigma_s^2\in \ATC.
\end{equation} In \cite{hebert2021decompositions} and \cite{hebert2021decompositions}, $\Itg$  is described in terms of the $\ev_\tau(\KC_{s_1}*\ldots*\KC_{s_k}).\vb_\tau$, where $s_1,\ldots,s_k\in \SCC_\tau$. However,  if $w\in \Wta$ and $w=s_1\ldots s_k=s_1'\ldots s_k'$ is a reduced expression of $w$, with $k=\ell_\tau(w)$ and $s_1,\ldots,s_k,s_1',\ldots,s_k'\in \SCC_\tau$, we might have $\KC_{s_1}\ldots \KC_{s_k}\neq \KC_{s_1'}\ldots \KC_{s_k'}$. We thus slightly modify the $\KC_s$ and define $\tKC_s$, for $s\in \SCC$, so that we have $\tKC_{s_1}\ldots \tKC_{s_k}= \tKC_{s_1'}\ldots \tKC_{s_k'}$. This enables us to define $\tKC_w$, for $w\in \Wta$. Let $\KCC_\tau=\bigoplus_{w\in \Wta} \tKC_w *\C(Y)_\tau$. We prove that $\KCC_\tau$ has a presentation very close to the Bernstein-Lusztig presentation of $\AC_\C$. 

In order to study $\Itg$, it is then convenient to describe it as a $\KCC_\tau$-module. However,  $\KCC_\tau$ is not contained in  $\AC_\C$. We thus extend the action of $\KCC_\tau\cap \AC_\C$ on $\Itg$ to an action of $\KCC_\tau$ on $\Itg$.

In subsection~\ref{subBLH_structure_K}, we study $\KCC_\tau$ and in subsection~\ref{subAction_KC_Itau}, we define an action of $\KCC_\tau$ on $\Itg$. In subsection~\ref{subExample_infinite_basis}, we prove that there can exist $\tau\in \UC_\C$ for which $\SCC_\tau$ is infinite.

\subsection{Bernstein-Lusztig-Hecke structure on $\KCC_\tau$}\label{subBLH_structure_K}

 For $s\in \SCC_\tau$, we set  \begin{equation}\label{eqDefinition_tKC}
\tKC_s=\KC_s+\sigma_s^2=F_s+Q_s^T\in \ATC
\end{equation}\index{$\tKC_s$} Note that if $\tau(\lambda)=1$ for all $\lambda\in Y$, we have $\SCC_\tau=\SCC$ and 
\begin{equation}\label{eqK)T_if_tau=1}
\tKC_s=T_s
\end{equation} for $s\in \SCC$.  

Let $s\in \SCC_\tau$.  For $\theta\in \C(Y)$, set $\tilde{\Omega}_s(\theta)=Q_s^T(\theta-{^s\theta})$. Then similarly as in Remark~\ref{remIH algebre dans le cas KM deploye}~\eqref{itPolynomiality_Bernstein_Lusztig} and \cite[Lemma 5.22]{hebert2018principal} we have \begin{equation}\tilde{\Omega}_s(\C[Y])\subset \C[Y]\text{ and }\tilde{\Omega}_s\big(\C(Y)_\tau\big)\subset \C(Y)_\tau.\end{equation}

The aim of this subsection is to prove the following proposition.

\begin{Proposition}\label{propBL_relations}
Let $\tau\in T_\C$.
\begin{enumerate}
\item Let $w\in \Wta$ and $w=s_1\ldots s_k$ be a reduced decomposition of $w$, with $k=\ell_\tau(w)$ and $s_1\ldots,s_k\in \SCC_\tau$. Then  \[\tKC_{w}=\tKC_{s_1}*\ldots*\tKC_{s_k}\in \ATC_\tau\]\index{$\tKC_w$} is well defined, independently of the choice of the reduced decomposition.

\item\label{itHecke_relations} Let  $s\in \SCC_\tau$ and $w\in \Wta$. Then    $\tKC_{s}*\tKC_{w}=\left\{\begin{aligned} & \tKC_{sw} &\mathrm{\ if\ }\ell_\tau(sw)=\ell(w)+1\\ & (\sigma_{s}^2-1) \tKC_{w}+\sigma_s^2 \tKC_{s w} &\mathrm{\ if\ }\ell_\tau(sw)=\ell_\tau(w)-1 .\end{aligned}\right . \ $

\item\label{itBL_relations}  For $\theta\in \C(Y)$ and $s\in \SCC_\tau$, we have \begin{equation}\label{eqBL_relation}
\theta*\tKC_s=\tKC_s*{^s\theta}+\tilde{\Omega}_s(\theta).
\end{equation}

\item The matrix $\big(\alpha_s(\alpha_r^\vee)\big)_{r,s\in \SCC_\tau}$ is a (possibly infinite) Kac-Moody matrix.

\item\label{itAlgebra_K} If $\tau\in \UC_\C$,  the space \[\KCC_\tau=\sum_{k\in \N, s_1\ldots,s_k\in \SCC_\tau} \KC_{s_1}*\ldots *\KC_{s_k}*\C(Y)_\tau=\bigoplus_{w\in \Wta}\tKC_w*\C(Y)_\tau\] is a subalgebra of $\ATC$ contained in  $\ATC_\tau$.

\end{enumerate}
\end{Proposition} 

In particular, $\bigoplus_{w\in \Wta} \tKC_w*\C(Y)_\tau$ is almost a Bernstein-Lusztig-Hecke algebra as defined in subsection~\ref{subIH algebras}. Note however that $\SCC_\tau$ can be infinite (and thus the $(\alpha_s)_{s\in \SCC_\tau}$ and  $(\alpha_{s}^\vee)_{s\in \SCC_\tau}$ are not  free), see Lemma~\ref{lemExample_infinite_SCC_tau}. The proof of this proposition is  a consequence of the lemmas of this subsection below.

\begin{Lemma}
Let $s\in \SCC_\tau$. Then: \begin{enumerate}
\item $\KC_s^2=-(1+\sigma_s^2)\KC_s$,

\item $\tKC_s^2=(\sigma_s^2-1)\tKC_s+\sigma_s^2$.
\end{enumerate} 
\end{Lemma}

\begin{proof}
By Lemma~\ref{lemReeder 4.3}, \eqref{eqDef_Ks} and \cite[Lemma 4.3]{hebert2018principal}, we have \[\begin{aligned}\KC_s^2 &=(F_s-\zeta_s)(F_s-\zeta_s)\\
 &=F_s^2-F_s(\zeta_s+{^s\zeta_s})+\zeta_s.{^s\zeta_s}\\ 
&=\zeta_s.({^s\zeta_s})-F_s(\zeta_s+{^s\zeta_s})+\zeta_s.{{^s\zeta_s}}\\
 &=(-F_s+\zeta_s)(\zeta_s+{^s\zeta_s})=-(1+\sigma_s^2)\KC_s.\end{aligned}\]
\end{proof}

\begin{Lemma}\label{lemBernstein_lusztig_relations}
Let $\theta\in \C(Y)$ and $s\in \SCC_\tau$. Then \[\theta*\tKC_s=\tKC_s*{^s\theta}+\tilde{\Omega}_s(\theta).\]
\end{Lemma}

\begin{proof}
By Lemma~\ref{lemReeder 4.3}, we have \[\begin{aligned}\theta*\tKC_s&=\theta*\left(F_s+Q_s^T\right)\\
&=F_s* {^s\theta}+\theta*Q_s^T \\
&=\left(F_s+Q_s^T\right){^s\theta}+Q_s^T(\theta-{^s\theta})\\
&= \tKC_s *{^s\theta} +\tilde{\Omega}_s(\theta).\end{aligned}\] 
\end{proof}

Following \cite{deodhar1989subgroups}, we define $\leq_\tau$ on $\Phi_{(\tau),+}^\vee$ as follows. If $\alpha^\vee,\beta^\vee\in \Phi^\vee_{(\tau),+}$, we write $\alpha^\vee\leq_{\tau} \beta^\vee$ if there exist $k\in \N$, $a\in \R^*_+$,  $\beta_1^\vee,\ldots,\beta_k^\vee\in \Phi_{(\tau),+}^\vee$ and $a_1,\ldots,a_k\in \R_+$ such that $\alpha^\vee= a\beta^\vee+\sum_{i=1}^k a_i \beta_i^\vee$.

Let $s\in \Wta\cap \RCC$. Then by \cite[Lemma 5.13 (2)]{hebert2018principal}, we have \begin{equation}\label{eqCharacterization_S_tau}s\in \SCC_\tau\text { if and only if }s.\big(\Phi_{(\tau),+}^\vee\setminus \{\alpha_s^\vee\}\big)=\Phi_{(\tau),+}^\vee\setminus\{\alpha_s^\vee\}.\end{equation}

\begin{Lemma}\label{lemBasis_root_subsystem}
Let \[\Sigma_\tau=\{\alpha^\vee\in \Phi_{(\tau),+}^\vee|\forall\beta^\vee\in \Phi_{(\tau),+}^\vee, \beta^\vee\leq_\tau \alpha^\vee\Rightarrow \beta^\vee=\alpha^\vee\}.\]  Then the map $s\mapsto \alpha_s^\vee$ from $\SCC_\tau$ to $\Sigma_\tau$ is well defined and is a bijection. Moreover, we have \begin{equation}\label{eqSigma_spans_Phi}\Phi^\vee_{(\tau),+}\subset \sum_{\alpha^\vee\in \Sigma_\tau}\N \alpha^\vee.\end{equation}
\end{Lemma}

\begin{proof}
This follows \cite[3. Step 2]{deodhar1989subgroups}.

Let $\alpha^\vee\in \Sigma_\tau$ (if such an element exists).  Let $s=r_{\alpha^\vee}\in \RCC_{(\tau)}$ and $\beta^\vee\in \Phi^\vee_{(\tau),+}$. We assume that $\gamma^\vee:=-s.\beta^\vee\in \Phi^\vee_{(\tau),+}$. Then $\gamma^\vee=-\beta^\vee+\alpha_s(\beta^\vee)\alpha_s^\vee$. As $\beta^\vee\in \Phi^\vee_+$, we necessarily have $\alpha_s(\beta^\vee)\in \N^*$. Therefore $\alpha_s^\vee=\frac{1}{\alpha_s(\beta^\vee)}(\beta^\vee+\gamma^\vee)$ and $\beta^\vee\leq_\tau \alpha_s^\vee$. By definition of $\Sigma_\tau$, we deduce $\beta=\alpha_s^\vee$. By \eqref{eqCharacterization_S_tau} we deduce that $s\in \SCC_\tau$.

 Conversely, let $s\in \SCC_{\tau}$. Let $\beta^\vee\in \Phi^\vee_{(\tau),+}$  be such that $\beta^\vee\leq_\tau \alpha_s^\vee$. By definition, we can write $\alpha_s^\vee = a\beta^\vee+\sum_{i=1}^k a_i \beta_i^\vee$, where $k\in \N$, $a\in \R^*_+$,  $(a_i)\in (\R_+)^k$ and $(\beta_i^\vee)\in \big(\Phi_{(\tau),+}\big)^k$. Then \[s.\beta^\vee=-\frac{1}{a}(\alpha_s^\vee+\sum_{i=1}^k a_is.\beta_i^\vee)\in \Phi^\vee_-\cap\Phi^\vee_{(\tau)}.\]  By \eqref{eqCharacterization_S_tau}, we deduce that $\beta^\vee=\alpha_s^\vee$ and thus $\alpha_s^\vee\in \Sigma_\tau$.    We proved the first part of the lemma. Then \eqref{eqSigma_spans_Phi} can be proved as  \cite[3 step 3]{deodhar1989subgroups}, using the fact that $\alpha(\beta^\vee)\in \Z$, for all $\alpha^\vee,\beta^\vee\in \Phi_{(\tau)}^\vee$.
 
 \end{proof}
 
 \begin{Lemma}\label{lemKac-Moody_matrix}
The matrix $\big(\alpha_s(\alpha_r^\vee)\big)_{r,s\in \SCC_\tau}$ is a Kac-Moody matrix.
 \end{Lemma}
 
 \begin{proof}
 Let $r,s\in\SCC_\tau$.  Suppose $\alpha_r(\alpha_s^\vee)\neq 0$. We have $r.\alpha_s^\vee=\alpha_s^\vee-\alpha_r(\alpha_s^\vee)\alpha_r^\vee$. By Lemma~\ref{lemBasis_root_subsystem}, $r.\alpha_s^\vee \not \leq_\tau \alpha_r^\vee$, which implies that $\alpha_r(\alpha_s^\vee)\leq 0$. 
   
   Suppose $\alpha_r(\alpha_s^\vee)=0$. Then by the case above, $\alpha_s(\alpha_r^\vee)\leq 0$. Suppose that $\alpha_s(\alpha_r^\vee)<0$. Then  \[\alpha_r^\vee=-rs.\alpha_r^\vee-\alpha_s(\alpha_r^\vee)\alpha_s^\vee\text{ and }\alpha_s^\vee=\frac{-1}{\alpha_s(\alpha_r^\vee)}(rs.\alpha_r^\vee+\alpha_r^\vee).\] If $rs.\alpha_r^\vee\in \Phi^\vee_+$, we deduce that $\alpha_s^\vee\leq_{\tau} \alpha_r^\vee$ and if $rs.\alpha_r^\vee$, we deduce that $\alpha_r^\vee\leq_\tau \alpha_s^\vee$. Since $\alpha_r^\vee,\alpha_s^\vee\in \Sigma_\tau$, this implies in both cases that $\alpha_r^\vee=\alpha_s^\vee$: we reach a contradiction.   Therefore $\alpha_s(\alpha_r^\vee)=0$ and the lemma follows.
 \end{proof}

If $a,b$ are two elements of a ring and $m\in \N$, we denote by $(a.b)^{*m}$ the product $a.b.a.\ldots$ with $m$ factors. 

\begin{Lemma}
Let $w\in \Wta$ and $w=r_1\ldots r_k= s_1\ldots s_k$ be two reduced writings of $w$, with $k=\ell_\tau(w)$ and $r_1,\ldots, r_k, s_1,\ldots,s_k\in \SCC_\tau$. Then \[\tKC_{r_1}\ldots \tKC_{r_k}=\tKC_{s_1}\ldots \tKC_{s_k}.\]
\end{Lemma}

\begin{proof}
Let $r,s\in \SCC_\tau$. Suppose that the order $m$ of $rs$ is finite.  Let $h=(\tKC_r *\tKC_s)^{*m}-(\tKC_s*\tKC_r)^{*m}\in \ATC$. We want to prove that $h=0$. Set $w_0(r,s)=(rs)^{*m}=(sr)^{*m}$. Using Lemma~\ref{lemReeder 4.3}, \eqref{eqDefinition_tKC} and \cite[Lemma 4.3]{hebert2018principal}, there exists a family \[(P_u)_{u\in \langle r,s\rangle}\in \left(\Z[\sigma_r,\sigma_s][x_{s,u},x_{r,u'}|u,u'\in \langle r,s\rangle]\right)^{\langle r,s\rangle}\] such that \begin{equation}
h=\sum_{u\in \langle r,s\rangle} F_u*P_u\left(\left({^v Q^T_s},^w Q^T_r \right)_{v,w\in \langle r,s\rangle}\right),
\end{equation}

where the $x_{s,u},x_{r,u}$ are indeterminates. 

Let $A_{r,s}=\begin{pmatrix}
2 & \alpha_s(\alpha_r^\vee)\\ \alpha_r(\alpha_s^\vee) & 2
\end{pmatrix}$. Then $A_{r,s}$ is  a Kac-Moody matrix by Lemma~\ref{lemKac-Moody_matrix} (it is actually a Cartan matrix since $\langle r,s\rangle$ is finite).  Let $X',Y',\alpha_r',\alpha_s',\alpha_r^\vee,\alpha_s^\vee$ be copies of $X,Y,\alpha_r,\alpha_s,\alpha_r^\vee,\alpha_s^\vee$.   Then $\SC'=(A_{r,s},X',Y',(\alpha_r',\alpha_s'),(\alpha_r'^\vee,\alpha_s'^\vee)\big)$ is a root generating system. 
Let $W'$ be the associated Weyl group. We add a prime to all the objects previously defined when they refer to the root generating system $\SC'$. Let $\tau'\in T_\C'$ be defined by $\tau'(\lambda)=1$ for all $\lambda\in Y'$. Then we have $\SCC'_{\tau'}=\{r,s\}\subset W'$. By \eqref{eqK)T_if_tau=1}, $\tKC_r'=T_r'\in\ATC'$ and $\tKC_s'=T_s'\in \ATC'$. Moreover $(T_r'*T_s')^{*m} -(T_s'*T_r')^{*m}=T_{w_0'}'-T_{w_0'}'=0$, where $w_0'=(rs)^{*m}=(sr)^{*m}\in W'$ is the element of maximal length of $W'$.  Therefore if $u\in W'$, we have
 \[P_u\big(({^v{Q_r'^T}},{^w Q_s'^T)_{v,w\in W'}}\big)=0.\]
  Let $Q'^\vee=\Z\alpha_r'^\vee\oplus \Z\alpha_s'^\vee$ and $\iota:Q'^\vee\rightarrow Q^\vee$ be the $\Z$-modules morphism defined by $\iota(\alpha_r'^\vee)=\alpha_r^\vee$ and $\iota(\alpha_s'^\vee)=\alpha_s^\vee$. Extend $\iota$ to a $\C$-algebra morphism  $\iota:\C[Q'^\vee]\rightarrow \C[Q^\vee]$. Then $\iota$ is compatible with the actions of $W'$ and $\langle r,s\rangle$. More precisely, if $\iota:W'\rightarrow \langle r,s\rangle$ is the group morphism defined by $\iota(r)=r$, $\iota(s)=s$,  we have $\iota\left({^w Q'^T_{t}}\right)={^{\iota(w)} Q^T_{t}},$ for $w\in W'$ and $t\in \{r,s\}$. 
 Therefore \[P_u\left(\left({\iota\left(^vQ'^T_r\right)}, \iota\left({^w Q'^T_s}\right)\right)_{v,w\in W'}\right)=P_u\left(\left({^vQ^T_r},{^w Q^T_s}\right)_{v,w\in \langle r,s\rangle}\right)=0,\] 
 for $u\in \langle r,s\rangle$ which proves that $h=0$ and hence $(\tKC_r*\tKC_s)^{*m}=(\tKC_s*\tKC_r)^{*m}$. By the word property (\cite[Theorem 3.3.1]{bjorner2005combinatorics}), we deduce the lemma.
\end{proof}

To conclude the proof of Proposition~\ref{propBL_relations}, it remains to prove (5). By (2) and (3), $ \KCC_\tau=\sum_{k\in \N, s_1,\ldots,s_k\in \SCC_\tau}\tKC_{s_1}*\ldots*\tKC_{s_k}*\C(Y)_\tau$ is a subalgebra of $\ATC$. Let $w\in \Wta$. Let $w=s_1\ldots s_k$ be a reduced writing of $w$, with $k\in \N$ and $s_1,\ldots,s_k\in \SCC_\tau$. Then by Lemma~\ref{lemReeder 4.3}, \begin{equation}\label{eqTriangular_writing_F}
 \tKC_{s_1}*\ldots*\tKC_{s_k}=(F_{s_1}+Q_{s_1}^T)\ldots (F_{s_k}+Q_{s_k}^T)= F_{w}+\sum_{v<_\tau w} F_v*\theta_v,  
\end{equation}for some $\theta_v\in \C(Y)$. For $h\in \ATC$, $h=\sum_{v\in W^v} T_v*\tilde{\theta}_v$, write \[\max \supp(h)=\{w\in W^v| \tilde{\theta}_w\neq 0\text{ and }\forall w'\in W^v, w'>w, \tilde{\theta}_{w'}=0\}.\] Then by \cite[Lemma 5.24]{hebert2018principal} and \eqref{eqTriangular_writing_F}, $\max \supp(\tKC_w)=\{w\}$, for $w\in \Wta$ and thus $(\tKC_{w})_{w\in \Wta}$ is a free family of $\ATC$.

By \eqref{eqDefinition_tKC} and (3), $\sum_{k\in \N,s_1,\ldots,s_k\in \SCC_\tau} \KC_{s_1}*\ldots*\KC_{s_k}*\C(Y)_\tau\subset \bigoplus_{w\in \Wta} \tKC_w*\C(Y)_\tau$. The converse inclusion is obtained similarly and thus $\KCC_\tau=\sum_{k\in \N,s_1,\ldots,s_k\in \SCC_\tau} \KC_{s_1}*\ldots*\KC_{s_k}*\C(Y)_\tau$.

 By \cite[Lemma 5.23]{hebert2018principal}, $\ATC_\tau*\KCC_\tau\subset \ATC_\tau$ for $s\in \SCC_\tau$. Therefore $\tKC_w\in \ATC_\tau$ for $w\in \Wta$ and $\KCC_\tau\subset \ATC_\tau$. This completes the proof of Proposition~\ref{propBL_relations}. $\square$

\begin{Remark}
In \cite[5.5]{hebert2018principal}, we assumed that $\tau\in \UC_\C$ and that $W_\tau=\Wta$. The assumption $W_\tau=\Wta$ is actually useless (without any change in the proofs). The assumption $\tau\in \UC_\C$ is however used, to ensure that $\max \supp(\tKC_{w})=\{w\}$ (see \cite[Lemma 5.24]{hebert2018principal}), for $w\in \Wta$ and thus to prove the freeness of $(\tKC_w)_{w\in \Wta}$. 
\end{Remark}

\subsection{Action of $\KCC_\tau$ on $\Itg$}\label{subAction_KC_Itau}

We now fix $\tau\in \UC_\C$. Let $\KCC_\tau=\bigoplus_{w\in \Wta} \tKC_{w} *\C(Y)_\tau\subset \ATC_\tau$. In this subsection, we extend the action of $\KCC_\tau\cap \AC_\C$ on $\Itg$ to an action of $\KCC_\tau$ on $\Itg$ (see Lemma~\ref{lemAction_K_Itg}).

Let $\JC_\tau=\{\theta\in \C(Y)_\tau|\tau(\theta)\neq 0\}$\index{$\JC_\tau$}.

\begin{Lemma}\label{lemLeft_multiplication_CY_AC}
Let $k\in \KCC_\tau$. Then there exists $\theta \in \C[Y]\setminus \JC_\tau$ such that $\theta*k\in \AC_\C$.
\end{Lemma}

\begin{proof}
 Let $w\in \Wta$. We assume that for all $v\in \Wta$ such that $v<_\tau w$, for all $\tilde{\theta}\in \C(Y)_\tau$, there exists $\theta\in \C[Y]\setminus \JC_\tau$ such that $\theta*\tKC_v*\tilde{\theta}\in \AC_\C$. Let $\tilde{\theta}\in \C(Y)_\tau$. Let $s\in \SCC_\tau$ be such that $v:=sw<_\tau w$.  Let $\theta_1\in \C[Y]\setminus \JC_\tau$ be such that $\theta_1*\tKC_v*\tilde{\theta}\in \AC_\C$. Let $\theta_2\in \C[Y]\setminus \JC_\tau$ be such that $\tKC_s*\theta\in \AC_\C$, which exists by Proposition~\ref{propBL_relations}\eqref{itAlgebra_K}. Set $\theta_3=(\theta_1*\theta_2)*{^s( \theta_1*\theta_2)}\in \C[Y]\setminus \JC_\tau$. Then ${^s\theta_3}=\theta_3$, thus $\tilde{\Omega}_s(\theta_3)=0$ and hence $\theta_3*\tKC_s=\tKC_s*\theta_3$. Therefore \[\theta_3* \tKC_{w}*\theta=\tKC_s*\theta_3*\tKC_v*\theta=\tKC_s *\theta_2 *{^s (\theta_1*\theta_2)}*\theta_1*\tKC_v\in \AC_\C.\] We deduce that for all $w\in \Wta$ and $\tilde{\theta}\in \C(Y)_\tau$, there exists $\theta\in \C[Y]\setminus \JC_\tau$ such that $\theta*\tKC_w*\tilde{\theta}\in \AC_\C$.
 
 Let  now $k\in\KCC_\tau$. Write $k=\sum_{w\in \Wta} \tKC_w *\tilde{\theta}_w$, with $(\tilde{\theta}_w)\in (\C(Y)_\tau)^{(\Wta)}$. For $w\in \Wta$, choose $\theta_w\in \C[Y]\setminus \JC_\tau$ such that $\theta_w*\tKC_w*\tilde{\theta}_w\in \AC_\C$. Set $\theta=\prod_{w\in \Wta|\tilde{\theta}_w\neq 0} \theta_w\in \C[Y]\setminus \JC_\tau$. Then $\theta*k\in \AC_\C$, which proves the lemma.
\end{proof}

\begin{Lemma}\label{lemInjectivity_multiplication_theta}
Let $\theta\in \C[Y]\setminus\JC_\tau$. Then the map $m_\theta:\Itg\rightarrow \Itg$ defined by $m_\theta(x)=\theta.x$, for $x\in \Itg$, is injective.
\end{Lemma}

\begin{proof}
Let $x\in I_\tau$. Write $x=\sum_{w\in W^v} a_w T_w.\vb_\tau$, where $(a_w)\in \C^{(W^v)}$. Let $w\in W^v$ be such that $a_w\neq 0$ and such that for all $v\in W^v$ such that $a_v\neq 0$, we have $v\not \geq w$. Then by \cite[Lemma 2.8]{hebert2018principal}, we have $\theta.x-a_w T_w*{^{w^{-1}}\theta}.\vb_\tau \in \bigoplus_{v\in W^v, v\not \geq w}\C T_v.\vb_\tau$. By \cite[Lemma 5.24]{hebert2018principal}, $w\in \Wta\subset W_\tau$. Therefore ${^{w^{-1}}\theta}.\vb_\tau=\tau(\theta).\vb_\tau\neq 0$, which proves that $\theta.x\neq 0$. Therefore $m_\theta$ is injective.
\end{proof}

\begin{Lemma}\label{lemAction_K_Itg}
We have the following properties:\begin{enumerate}

\item  $\Itg=(\KCC_\tau\cap \AC_\C).\vb_\tau$,

\item $\KCC_\tau\cap \AC_\C$ stabilizes $\Itg$ and the action of $\KCC_\tau\cap \AC_\C$ on $\Itg$ extends uniquely to an action of $\KCC_\tau$ on $\Itg$. This action is as follows. Let $x\in \Itg$ and $k\in \KCC_\tau$. Write $x=h.x$, with $h\in \AC_\C\cap \KCC_\tau$. Write $k*h=\sum_{w\in \Wta} \tKC_w \tilde{\theta}_w$, with $(\theta_w)\in (\C(Y)_\tau)^{(\Wta)}$. Then: \begin{equation}\label{eqAction_K_Itg}
k.x=k*h.\vb_\tau=\sum_{w\in \Wta} \tau(\tilde{\theta})\tKC_w (\tau).\vb_\tau.
\end{equation}
\end{enumerate}
\end{Lemma}

\begin{proof}
 By \cite[Theorem 5.27]{hebert2018principal}, we have $\Itg=\ev_\tau(\KCC_\tau).\vb_\tau$.

 For $w\in \Wta$, write $\tKC_{w}=\sum_{v\in W^v} T_w *\theta_{v,w}$, with $\theta_{v,w}\in \C(Y)$ for $v\in W^v$ and $\{v\in W^v|\theta_{v,w}\neq 0\}$ finite. By \cite[Lemma 5.23]{hebert2018principal}, we can write $\theta_{v,w}=\frac{f_{v,w}}{g_{v,w}}$, where $f_{v,w},g_{v,w}\in \C[Y]$ and $\tau(g_{v,w})\neq 0$ for $v\in W^v$ such that $f_{v,w}\neq 0$. For $w\in \Wta$, set $g_w=\prod_{v\in W^v|f_{v,w}\neq 0} g_{v,w}\in \C[Y]$. Then  \begin{equation}\label{eqPolynomial_basis}
\ev_\tau(\tKC_w)=\ev_\tau(\frac{1}{\tau(g_w)}\tKC_w g_w)\text{ and }\frac{1}{\tau(g_w)}\tKC_w g_w\in \KCC_\tau\cap \AC_\C.
\end{equation} Therefore \begin{equation}
\Itg=(\KCC_\tau\cap \AC_\C).\vb_\tau,
\end{equation}

which proves (1).

(2)  By Proposition~\ref{propBL_relations}\eqref{itAlgebra_K}, $\KCC_\tau\cap \AC_\C$ is a subalgebra of $\AC_\C$, which proves that $\KCC_\tau\cap \AC_\C$ stabilizes $\Itg$.  For $w\in \Wta$, set $L_w=\KC_w*g_w\in \AC_\C\cap \KCC_\tau$ and set \[\LCC_\tau=\bigoplus_{w\in \Wta} L_w*\C[Y]\subset \AC_\C\cap \KCC_\tau.\]

Let $\tilde{\tau}:\C(Y)_\tau\rightarrow \C$ be the algebra morphism extending $\tau:\C[Y]\rightarrow \C$. Let $\tilde{M}=\mathrm{Ind}_{\C(Y)_\tau}^{\KCC_\tau}(\tilde{\tau})$ be the representation of $\KCC_\tau$ obtained by inducing $\tilde{\tau}$ to $\KCC_\tau$. Then \[\tilde{M}=\bigoplus_{w\in \Wta} \tKC_w.\vbt,\] where $\vbt\in \tilde{M}$ is such that $\theta.\vbt=\tilde{\tau}(\theta).\vbt$, for $\theta\in \C(Y)_\tau$. Let $\psi:\tilde{M}\rightarrow \Itg$ be defined by $\psi(k.\vbt)=k.\vb_\tau$, for $k\in \LCC_\tau$.  Let us prove that $\psi$ is well defined. Let $k\in \LCC_\tau$ be such that $k.\vbt=0$. We can write $k=\sum_{w\in W^v} L_w*\theta_w$, for some $(\theta_w)\in \C[Y]^{(\Wta)}$. Then $ \tau(\theta_w)=0$ for all $w\in \Wta$ and thus $k.\vb_\tau=0$, which proves that $\psi:\LCC_\tau.\vbt\rightarrow \Itg$ is well defined. Moreover, $\tilde{M}=\bigoplus_{w\in \Wta}\C \tKC_w.\vbt=\bigoplus_{w\in \Wta} \C L_w.\vbt=\LCC_\tau.\vbt$. Thus $\psi:\tilde{M}\rightarrow \Itg$ is well defined. Then $\psi$ is an isomorphism of $\LCC_\tau$-modules. We equip $\Itg$ with the structure of a $\KCC_\tau$-module by setting $k\odot x=\psi\big(k.\psi^{-1}(x)\big)$ for $k\in \KCC_\tau$ and $x\in \Itg$. 

Let $w\in \Wta$ and $\theta\in \C(Y)_\tau$. Then \begin{equation}\label{eqtKC_theta_vtau}
(\tKC_w*\theta)\odot \vb_\tau=\psi(\tKC_w*\theta.\vbt)=\tau(\theta)\psi(\tKC_w.\vbt)=\tau(\theta)\tKC_w\odot\vb_\tau.
 \end{equation} By applying this to $\theta=g_w$, we deduce that \begin{equation}\label{eqtKC_vtau}
\tKC_w\odot\vb_\tau=\frac{1}{\tau(g_w)}\tKC_w*g_w\odot\vb_\tau=\frac{1}{\tau(g_w)}\tKC_w*g_w.\vb_\tau=\frac{1}{\tau(g_w)}\ev_\tau(\tKC_w*g_w).\vb_\tau=\tKC_w(\tau).\vb_\tau.
\end{equation} Combining \eqref{eqtKC_theta_vtau} and \eqref{eqtKC_vtau} yields \eqref{eqAction_K_Itg}.

Let $k\in \AC_\C\cap \KCC_\tau$. Write $k=\sum_{w\in \Wta} \tKC_w*\theta_w$, with $(\theta_w)\in (\C(Y)_\tau)^{(\Wta)}$. For $w\in \Wta$, write $\tKC_w=\sum_{v\in W^v}T_v*\theta_{v,w}$, with $(\theta_{v,w})_{v\in W^v}\in (\C(Y)_\tau)^{(W^v)}$. Then $k=\sum_{v\in W^v}\sum_{w\in \Wta} T_v*\theta_{v,w}*\theta_w$ and $\sum_{w\in\Wta} \theta_{v,w}*\theta_{w}\in \C[Y]$ for every $v\in W^v$. Then \[\begin{aligned} k.\vb_\tau &=\sum_{v\in W^v} \tau(\sum_{w\in \Wta} \theta_{v,w}*\theta_w) T_v.\vb_\tau \\
&=\sum_{v\in W^v}  \sum_{w\in \Wta}\tau(\theta_{v,w}*\theta_w) T_v.\vb_\tau\\
&= \sum_{w\in \Wta} \tau(\theta_w)\ev_\tau(\sum_{v\in W^v} T_v*\theta_{v,w}).\vb_\tau\\
&= \sum_{w\in \Wta} \tau(\theta_w) \ev_\tau(\tKC_w).\vb_\tau =k\odot \vb_\tau,\end{aligned}\] by \eqref{eqAction_K_Itg}. Therefore $\odot$ extends the action $.$ of  $\AC_\C\cap \KCC_\tau$ on $\Itg$ to an action of $\KCC_\tau$ on $\Itg$.

 We now prove  the uniqueness of such an action. Let $\boxdot$ be an action of $\KCC_\tau$ on $\Itg$ such that $k\boxdot x=k.x$ for all $k\in \AC_\C\cap \KCC_\tau$ and $x\in  \Itg$.  Let $k\in \KCC_\tau$. Let $\theta\in \C[Y]\setminus\JC_\tau$ be such that $\theta*k\in \AC_\C$, which exists by Lemma~\ref{lemLeft_multiplication_CY_AC}. Then \[(\theta*k).x=(\theta*k)\boxdot x=\theta\boxdot (k\boxdot\vb_\tau)=\theta.(k\boxdot\vb_\tau).\] By Lemma~\ref{lemInjectivity_multiplication_theta} we deduce that $k\boxdot x=m_\theta^{-1}(\theta*k.x)$. This proves that $\boxdot=\odot$ and concludes the proof of this lemma.

\end{proof}

\subsection{Examples of $\tau$ for which $\SCC_\tau$ is infinite}\label{subExample_infinite_basis}

The set $\Phi^\vee$ is a root system in the sense of \cite[4 Remark]{moody1989infinite}. Then $\Phi_{(\tau)}^\vee$ is a subroot system of $\Phi^\vee$ in the sense of \cite[6]{moody1989infinite}. It is proved in \cite[Example 1]{moody1989infinite} that the bases of a subroot system of a root system admitting a finite basis need not be finite. However, we do not know if the root system given in this example is of the form $\Phi^\vee_{(\tau)}$, for some $\tau\in T_\C$. Suppose that  $\sigma_s=\sigma_s'$ for all $s\in\SCC$, then $\Phi^\vee_{(\tau)}=\{\alpha^\vee\in \Phi^\vee|\tau(\alpha^\vee)=1\}$. Thus if $\tau\in T_\C$,   $\Phi^\vee_{(\tau)}$ is  \textbf{closed} in the sense that for all $\alpha^\vee,\beta^\vee\in \Phi^\vee_{(\tau)}$, if $\alpha^\vee+\beta^\vee\in \Phi^\vee$, then $\alpha^\vee+\beta^\vee\in \Phi_{(\tau)}^\vee$. We may ask whether this property ensures the finiteness of $\SCC_\tau$. We prove below that this is not the case.

For $w\in W^v$, set $N_{\Phi^\vee}(w)=\{\alpha^\vee\in \Phi^\vee_+|w.\alpha^\vee\in \Phi^\vee_-\}$\index{$N_{\Phi^\vee}(w)$}.

\begin{Lemma}\label{lemExample_infinite_SCC_tau}
Let $A=(a_{i,j})_{i,j\in \llbracket 1,4\rrbracket}$ be an invertible Kac-Moody matrix such that for all $i,j\in \llbracket 1,4\rrbracket$, $a_{i,j}\leq -2$, for all $(i,j)\in \llbracket 1,4\rrbracket\setminus \{(4,3)\}$, $a_{i,j}$ is even and $a_{4,3}$ is odd. We assume that $\A=\bigoplus_{i=1}^4 \R\alpha_i^\vee$, which is possible since $A$ is invertible by \cite[1.1]{kac1994infinite}. Let $\htt:\A\rightarrow \R$ be defined by $\htt(\sum_{i=1}^4 n_i\alpha_i^\vee)=\sum_{i=1}^4 n_i$, for $(n_i)\in \R^4$. Let $q$ be a prime power and set $\sigma_i=\sigma_i'=\sqrt{q}$ for $i\in \llbracket 1,4\rrbracket$. Let $\tau: Y\rightarrow \{-1,1\}$ be defined by $\tau(\lambda)=(-1)^{\htt(\lambda)}$ for $\lambda\in Y$. Let $W'=\langle
 r_1,r_2\rangle \subset W^v$. Then $\{w r_3r_4r_3 w^{-1}|w\in W'\}\subset \SCC_\tau$. In particular, $\SCC_\tau$ is infinite.
\end{Lemma}

\begin{proof}
Let $w\in W'$ and $v=wr_3r_4r_3w^{-1}$. Then by \cite[1.3.14 Lemma]{kumar2002kac}, $wr_3.\alpha_4^\vee\in N_{\Phi^\vee}(v)$ and $N_{\Phi^\vee}(v)\setminus\{wr_3.\alpha_4^\vee\}\subset \{u.\alpha_i^\vee|u\in W^v, i\in \llbracket 1,3\rrbracket\}$. We have $r_3.\alpha_4^\vee=\alpha_4^\vee-a_{4,3}\alpha_3^\vee$. Moreover if $i\in \llbracket 1,2\rrbracket$, $r_i(\alpha_4^\vee-a_{4,3}\alpha_3^\vee+2Q^\vee)\subset \alpha_4^\vee-a_{4,3}\alpha_3^\vee+2Q^\vee$. In particular, $wr_3.\alpha_4^\vee\in \alpha_4^\vee-a_{4,3}\alpha_3^\vee+2Q^\vee$ and $\tau(wr_3.\alpha_4^\vee)=1$. Now if $i\in \llbracket 1,3\rrbracket$ and $j\in \llbracket 1,4\rrbracket$, $r_j.(\alpha_i^\vee+2Q^\vee)\subset \alpha_i^\vee+2Q^\vee$. Therefore $\tau(u.\alpha_i^\vee)=-1$ if $u\in W^v$. In particular, $N_{\Phi^\vee}(v)\cap \Phi_{(\tau)}=\{wr_3.\alpha_4^\vee\}=\{\alpha_v^\vee\}$ and by \cite[Lemma 5.13 (2)]{hebert2018principal}, $v\in \SCC_\tau$. Using \cite[1.3.21 Proposition]{kumar2002kac} we deduce that $\SCC_\tau$ is infinite.
\end{proof}

\section{Kato's irreducibility criterion}\label{secKato_s_irreducibility_criterion}

Let $\tau\in \UC_\C$. Let $x\in \Itg$. Write $x=\sum_{w\in \Wta} a_w \tKC_{w}.\vb_\tau$, with $(a_w)\in \C^{(W^v)}$. Let $\supp(x)=\{w\in \Wta|a_w\neq 0\}$ and  $\ell_\tau(x)=\max\{\ell_\tau(w)|w\in S\}$.  We set $\ell_\tau(0)=-\infty$.

Let $x\in \Itg$. We define $\ord_\tau(x)$ as the minimum of the $k\in \N$ such that for all $\theta_1,\ldots,\theta_k\in \JC_\tau$, we have $\theta_1\ldots\theta_k.x=0$. We will see in Lemma~\ref{lemLength_decreasing_multiplication_theta} that $\ord_\tau(x)\in \N$ for $x\in \Itg$. The aim of this subsection is to prove that $I_\tau(\tau,\Wta):=\Itg\cap I_\tau(\tau)=\C\vb_\tau$ (see Theorem~\ref{thmWeight_space}). We actually prove that for all $x\in \Itg$, $\ord_\tau(x)=\ell_\tau(x)+1$. We then deduce Kato's irreducibility criterion (see Corollary~\ref{corKatos_irreducibility_criterion}).

For $s\in \SCC_\tau$, we set \[\sigma_{s,\tau}''=\frac{1}{2}\big((\sigma_s^2-1)+\sigma_s(\sigma'_s-\sigma_s'^{-1})\tau(\alpha_s^\vee)\big).\]\index{$\sigma_{s,\tau}''$}  When $\sigma_s=\sigma_s'$, then $\tau(\alpha_s^\vee)=1$ and $\sigma_{s,\tau}''=\sigma_s^2-1$.

As we assumed that $|\sigma_s|,|\sigma_s'|>1$, for all $s\in \SCC$, we have $\sigma_{s,\tau}''\neq 0$ for every $s\in \SCC_\tau$. Indeed, we have $\sigma_{s,\tau}''=0$  if and only if $\sigma_s=\sigma'_s=\pm 1$ or $\sigma'_s\neq \sigma_s$ and $\sigma_s'=-\sigma_s^{-1}$. We have $\SCC_\tau\subset \RCC_{(\tau)}$ and thus $\tau(\alpha_s^\vee)\in \{-1,1\}$. If $\tau(\alpha_s^\vee)=-1$, then $\sigma_s\neq \sigma_s'$ and  $\sigma_s^2-1+\sigma_s(\sigma_s'-\sigma_s'^{-1})\tau(-\alpha_s^\vee)=\sigma_s^2-1-\sigma_s(\sigma_s'-\sigma_s'^{-1})=0$ if and only if $\sigma_s'\in \{\sigma_s,-\sigma_s^{-1}\}$ if and only if $\sigma_s'=-\sigma_s^{-1}$. 

\begin{Lemma}\label{lemFormula_omega_tilde_vtau}
Let $s\in \SCC_\tau$ and $\lambda\in Y$. Then \begin{equation}\label{eqFormula_tau_Omega_s}
\tau\big(\tilde{\Omega}_s(Z^\lambda)\big)=\tau(\lambda)\sigma_{s,\tau}''\alpha_s(\lambda).
\end{equation} 
\end{Lemma}

\begin{proof}

We have \[\tilde{\Omega}_s(Z^\lambda)=\big((\sigma_s^2-1)+\sigma_s(\sigma_s'-\sigma_s'^{-1})Z^{-\alpha_s^\vee}\big) \frac{1-Z^{-\alpha_s(\lambda)\alpha_s^\vee}}{1-Z^{-2\alpha_s^\vee}}Z^{\lambda}.\] We have  $\tau(\alpha_s^\vee)\in \{-1,1\}$. If $\alpha_s(\lambda)$ is odd, then $\sigma_s=\sigma_s'$ and $\tau(\alpha_s^\vee)=1$ (since the denominator of $Q_s^T$ is then $1-Z^{-\alpha_s^\vee}$). We then conclude with  Remark~\ref{remIH algebre dans le cas KM deploye}~(\ref{itPolynomiality_Bernstein_Lusztig}). If $\alpha_s(\lambda)$ is even, a computation similar to the one of Remark~\ref{remIH algebre dans le cas KM deploye}~(\ref{itPolynomiality_Bernstein_Lusztig}) enables to prove~\eqref{eqFormula_tau_Omega_s}.

\end{proof}

For $h\in \KCC_\tau$, $h=\sum_{w\in \Wta} \tKC_w*\theta_w$, with $(\theta_w)\in \C^{(\Wta)}$ we set $\ell_\tau(h)=\max\{\ell_\tau(w)|\theta_w\neq 0\}$.

\begin{Lemma}\label{lemComputation_terms ordern-1_commutation} 
Let $w\in \Wta$. Fix a reduced writing $w=s_1\ldots s_k$ of $w$, with $k=\ell_\tau(w)$ and $s_1,\ldots,s_k\in \SCC_\tau$. For $i\in \llbracket 1,k\rrbracket$, set $v_i=s_1\ldots s_{i-1}$, $\tilde{v}_i=s_{i+1}\ldots s_k$. Let $E=\{i\in \llbracket 1,k\rrbracket|\ell_\tau(v_i\tilde{v}_i)=k-1\}$. Let $\theta\in \C(Y)$.  Then there exists $h\in \KCC_\tau$ such that $\ell_\tau(h)\leq k-2$ and 
\begin{equation}\label{eqBL_commutation_order_k-1_ATC}
\theta*\tKC_w=\tKC_w*{{^{w^{-1}}\theta}}+\sum_{i\in E}\tKC_{v_i\tilde{v_i}} *{^{\tilde{v}_i^{-1}}\tilde{\Omega}_{s_i}({^{v_i^{-1}}\theta}}) +h.
\end{equation} In particular if $\theta\in \C(Y)_\tau$ we have 
\begin{equation}\label{eqBL_commutation_order_k-1_Itau}
\theta.\tKC_w.\vb_\tau=\tau(\theta) \tKC_w.\vb_\tau+\tau(\lambda)\sum_{i\in E} \sigma_{s_i,\tau}''\alpha_{s_i}(v_i^{-1}.\theta) \tKC_{v_i\tilde{v}_i}.\vb_\tau+h.\vb_\tau.
\end{equation}
\end{Lemma}

\begin{proof}
We prove it by induction on $k$. If $k=0$, this is clear. We assume $k\geq 1$. Set $w'=ws_k$ and $\tilde{v}_i'=v_is_k$ for $i\in \llbracket 1,k-1\rrbracket$. Assume that we can write \[\theta*\tKC_{w'}=\tKC_{w'}*{^{w'^{-1}}\theta}+\sum_{i=1}^{k-1} \tKC_{v_i\tilde{v'_i}} *{^{\tilde{v'}_i^{-1}}\tilde{\Omega}_{s_i}(^{v_i^{-1}}\theta}) +h',\] where $h'\in \ATC$ and $\ell_\tau(h')\leq k-3$.
 By Proposition~\ref{propBL_relations}~\eqref{itBL_relations}, we have  \begin{equation}\label{eqBL_commutation_order_k-1_ATC_temporary}\begin{aligned} \theta*\tKC_w  &= \big(\tKC_{w'}*{^{w'^{-1}}\theta}+\sum_{i=1}^{k-1} \tKC_{v_i\tilde{v'_i}} *{^{\tilde{v'}_i^{-1}}\tilde{\Omega}_{s_i}(^{v_i^{-1}}\theta}) \big)*\tKC_{s_k}
\\ &=\tKC_{w}*{^{w^{-1}}\theta} +\tKC_{w'}*\tilde{\Omega}_{s_k}({^{w'^{-1}}\theta}) +\sum_{i=1}^{k-1}\left( \tKC_{v_i\tilde{v_i}} *{^{\tilde{v}_i^{-1}}\tilde{\Omega}_{s_i}({^{v_i^{-1}}\theta}})\right)+h'*\tKC_{s_k}+h''\\
&=\tKC_{w}*{^{w^{-1}}\theta} +\sum_{i=1}^{k} \tKC_{v_i\tilde{v_i}} *{^{\tilde{v}_i^{-1}}\tilde{\Omega}_{s_i}({^{v_i^{-1}}\theta}})+ +h'*\tKC_{s_k}+h'',\end{aligned}\end{equation} where \[h'' =\sum_{i=1}^{k-1}\tKC_{v_i\tilde{v_i'}} *\tilde{\Omega}_{s_k}\left({^{\tilde{v}_i^{-1}}\tilde{\Omega}_{s_i}({^{v_i^{-1}}\theta}})\right)\in \bigoplus_{v\in \Wta|\ell_\tau(v)\leq k-2}  \tKC_{v}*\C(Y)_\tau. \] Moreover by Proposition~\ref{propBL_relations}, $\ell_\tau(h'*\tKC_{s_k})\leq k-1$. By eliminating the $i$ such that $\ell_\tau(v_i\tilde{v}_i)<k-1$ in \eqref{eqBL_commutation_order_k-1_ATC_temporary}, we deduce   \eqref{eqBL_commutation_order_k-1_ATC}.

Using Lemma~\ref{lemFormula_omega_tilde_vtau}, we deduce \eqref{eqBL_commutation_order_k-1_Itau}.

\end{proof}

\begin{Lemma}\label{lemLength_decreasing_multiplication_theta}
Let $x\in \Itg$, $\theta\in \JC_\tau$ and $w\in W^v$. Then $\ell_\tau(\theta.x)\leq \ell_\tau(x)-1$. In particular we have $\ord_\tau(x)\leq \ell_\tau(x)+1$ for all $x\in \Itg$.
\end{Lemma}

\begin{proof}
This follows from  \eqref{eqBL_commutation_order_k-1_Itau}. 
\end{proof}

\begin{Lemma}\label{lemEquality_length_order_simple_product}
We assume that there exists $\rho\in \C^*$ such that  \begin{equation}\label{eqDef_rho}\sigma''_{s,\tau}\in \rho \R^*_+,\end{equation} for all $s\in \SCC$. Let $w\in \Wta$. Then $\ord_\tau(\tKC_w.\vb_\tau)=\ell_\tau(w)+1$. 
 \end{Lemma}
 
 \begin{proof}
  Pick $\lambda\in C^v_f$ and set $\theta=Z^\lambda-\tau(\lambda)\in \JC_\tau$. For $v\in \Wta$, we define $C_v\in \C$ by \[\theta^{\ell_\tau(v)}.\vb=C_v.\vb_\tau,\] which is well defined by Lemma~\ref{lemLength_decreasing_multiplication_theta}. Fix a reduced writing $v=s_1\ldots s_k$ of $w$, with $k=\ell_\tau(w)$ and $s_1,\ldots,s_k\in \SCC_\tau$. We want to prove that $C_w\neq 0$. We assume that for all $v\in \Wta$ such that $\ell_\tau(v)<\ell_\tau(w)$, we have $C_v\in (\tau(\lambda)\rho)^{\ell_\tau(v)}\R^*_+$.  For $i\in \llbracket 1,k\rrbracket$, set $v_i=s_1\ldots s_{i-1}$, $\tilde{v}_i=s_{i+1}\ldots s_k$. Let $E=\{i\in \llbracket 1,k\rrbracket|\ell_\tau(v_i\tilde{v}_i)=k-1\}$. 
 Then by Lemma~\ref{lemComputation_terms ordern-1_commutation} and \eqref{eqFormula_tau_Omega_s}, we have \[\begin{aligned}C_w &=\sum_{i\in E} \tau\big(\tilde{\Omega}_{s_i}(^{v_i^{-1}}\theta)\big)C_{v_i\tilde{v}_i}\\
&= \tau(\lambda)\sum_{i\in E}\alpha_{s_i}(\lambda) \sigma_{s_i,\tau}'' v_i.\alpha_{s_i}(\lambda) C_{v_i\tilde{v}_i}.  \end{aligned} \] If $i\in E$, we have $\ell(v_is_i)>\ell(v_i)$ and thus by \cite[1.3.13 Lemma]{kumar2002kac}, $v_i.\alpha_{s_i}\in \Phi_+$. Therefore $v_i.\alpha_{s_i}(\lambda)\in \Z_{>0}$. We deduce that $C_w\in (\tau(\lambda)\rho)^{k}\R^*_+$. In particular, $C_w\neq 0$ and $\ord_\tau(\tKC_w .\vb_\tau)\geq \ell_\tau(w)+1$. By Lemma~\ref{lemLength_decreasing_multiplication_theta} we deduce that  $\ord_\tau(\tKC_w .\vb_\tau)=\ell_\tau(w)+1$.

 \end{proof}

\begin{Lemma}\label{lemEasy_inequality_multiplication_KC_s}
Let $x\in \Itg$ and $s\in \SCC_\tau$.  Then $\ord_\tau(\tKC_s.x)\leq \ord_\tau(x)+1$.
\end{Lemma}

\begin{proof}
We prove it by induction on $\ord_\tau(x)$. If $\ord_\tau(x)=0$, then $x=0$ and it is clear.  We assume that $\ord_\tau(x)>0$ and that  for all $y\in \Itg$ such that $\ord_\tau(y)<\ord_\tau(x)$, we have $\ord_\tau(\tKC_s.y)\leq \ord_\tau(y)+1$. Let $\theta\in\JC_\tau$. By Proposition~\ref{propBL_relations}~\eqref{itBL_relations}, we have \[\theta.\tKC_s.x=\tKC_s.{^s\theta}.x+\tilde{\Omega}_s(\theta).x,\] for $\theta\in \JC_\tau$. Moreover $\ord_\tau({^s\theta}.x)< \ord_\tau(x)$ and  thus $\ord_\tau(\tKC_s.{^s\theta}.x)\leq \ord_\tau({^s\theta}.x)+1\leq_\tau \ord_\tau(x)$ by assumption. Moreover,  $\ord_\tau\left(\tilde{\Omega}_s(\theta).x\right)\leq \ord_\tau(x)$, therefore $\ord_\tau(\theta.\tKC_s.x)\leq \ord_\tau(x)$ for every $\theta\in \JC_\tau$  and thus $\ord_\tau(\tKC_s.x)\leq\ord_\tau(x)+1$. Lemma follows.

\end{proof}

Let  $x\in \Itg\setminus\{0\}$. Write $x=\sum_{w\in \Wta}a_w\tKC_w.\vb_\tau$. Let $M=\{w\in \supp(x)|\ell_\tau(w)=\ell_\tau(x)\}$, $\NC_\tau=|M|$ and $\LTI(x)=\sum_{w\in M} a_w \tKC_w.\vb_\tau$.

\begin{Theorem}\label{thmWeight_space}
Let $\tau\in \UC_\C$. We assume  that \eqref{eqDef_rho} is satisfied and that $|\sigma_s|,|\sigma_s'|>1$, for all $s\in \SCC_\tau$.  Let $x\in \Itg\setminus\{0\}$. Then $\ord_\tau(x)=\ell_\tau(x)+1$. In particular, $\Itg\cap I_\tau(\tau)=\C \vb_\tau$ and \cite[Conjecture 5.16]{hebert2021decompositions} is true.
\end{Theorem}

\begin{proof}
If $\NC_\tau(x)=1$, this is Lemma~\ref{lemEquality_length_order_simple_product} and if $\ell_\tau(x)=0$, this is clear. We now assume that $\NC_\tau(x)\geq 2$ and that for all $y\in \Itg\setminus\{0\}$ such that $\ell_\tau(y)<\ell_\tau(x)$ or $\NC_\tau(y)<\NC_\tau(x)$, we have $\ord_\tau(y)=\ell_\tau(y)+1$. Let $M=\supp\big(\LTI(x)\big)$. There are two cases:\begin{enumerate}
\item there exists $s\in \SCC$ such that for all $w\in M$, $sw<_{\tau} w$,

\item  for all $s\in \SCC$, there exists $w\in M$ such that $sw>_\tau w$.
\end{enumerate}

We first assume (1). Then we can write\[x=\tKC_s.x'+x'',\] with $\ell_\tau(x')=\ell_\tau(x)-1$, $\ell_\tau(x'')\leq \ell_\tau(x')$ and $sv>_\tau v$ for all $v\in \supp(x')$. Let $\theta\in \JC_\tau$. Then \begin{equation}\label{eqTheta_KC} \theta.\tKC_s.x=\tKC_s*{^{s}\theta}.x'+\tilde{\Omega}_s(\theta).x'+\theta.x''.\end{equation}

By Lemma~\ref{lemLength_decreasing_multiplication_theta}, we have $\ell_\tau({^s\theta}.x') \leq \ell_\tau(x')-1$. For all $v\in \supp({^s\theta}.x')$, if $sv<_\tau v$, then $\ell_\tau(\tKC_s*\tKC_v.\vb_\tau)\leq \ell_\tau(v)\leq \ell_\tau({^s\theta}.x')\leq \ell_\tau(x')-1$ by Proposition~\ref{propBL_relations}~\eqref{itHecke_relations}. Therefore for all $v\in \supp({\tKC_s*^s\theta}.x)$ such that $\ell_\tau(v)=\ell_\tau(x')$, we have $sv<_\tau v$. Therefore  \begin{equation}\label{eqSupport}\supp\big(\LTI(\tKC_s*{^s\theta}*x')\big)\cap \supp\big(\LTI(x')\big)=\emptyset.\end{equation} Take $\lambda\in Y$ such that $\alpha_s(\lambda)=1$ and set $\theta=Z^\lambda-\tau(\lambda)\in \JC_\tau$. Then by Lemma~\ref{lemFormula_omega_tilde_vtau}, $\tau\big(\tilde{\Omega_s}(\theta)\big)=\tau(\lambda)\sigma_{s,\tau}''\alpha_s(\lambda)=\tau(\lambda)\sigma_{s,\tau}''$ and thus by combining \eqref{eqTheta_KC} and \eqref{eqSupport}, we obtain that $\ell_\tau(\theta.\tKC_s.x)=\ell_\tau(x')$. By assumption we deduce that $\ord_\tau(\theta.\tKC_s.x)=\ell_\tau(x')+1$. Consequently $\ord_\tau(\tKC_s.x)\geq \ell_\tau(x)+1$ and by Lemma~\ref{lemLength_decreasing_multiplication_theta} we deduce that $\ord_\tau(\tKC_s.x)=\ell_\tau(x)+1$.

We now assume that we are in the case (2). Let $w\in \supp\big((\LTI(x)\big)$ and $s\in \SCC$ be such that $sw<_\tau w$. Set $x'= \tKC_s.x$. Then by assumption we have $\ell_\tau(x')=\ell_\tau(x)+1$ and  thus \[\supp\big(\LTI(x')\big)\subset s.\supp\left(\LTI(x)\right)\setminus \{sw\}.\] In particular $\NC_\tau(x')<\NC_\tau(x)$. By our induction assumption we deduce that $\ord_\tau(x')=\ell_\tau(x')+1=\ell_\tau(x)+2$. By Lemma~\ref{lemEasy_inequality_multiplication_KC_s} we have that $\ord_\tau(x)\geq \ord_\tau(x')-1\geq \ell_\tau(x)+1$. By Lemma~\ref{lemLength_decreasing_multiplication_theta} we deduce that $\ell_\tau(x)+1=\ord_\tau(x)$. This completes the proof of the theorem.

\end{proof}

Note that the above proof can be simplified in the case where $\Wta$ is finite. In this case, denote by $w_0$, the maximal element of $\Wta$ (for the Bruhat order $\leq_\tau$). Let $x\in \Itg\setminus\{0\}$. Let $w\in \supp\big(\LTI(x)\big)$. Set $y=\tKC_{w_0 w^{-1}}.x$. Then $\supp(\tKC_{w_0w^{-1}}.x)\ni w_0$. By Lemma~\ref{lemEquality_length_order_simple_product}, $\ord_\tau(\tKC_{w_{0}w^{-1}}.x)=\ell_\tau(w_0)+1$. By Lemma~\ref{lemEasy_inequality_multiplication_KC_s}, we deduce that \[\ord_\tau (x)\geq \ord_\tau(y)-\ell(w_0w^{-1})=\ell_\tau(w_0)+1-\big(\ell(w_0)-\ell(w^{-1})\big)=\ell(w)+1.\] By Lemma~\ref{lemLength_decreasing_multiplication_theta}, we deduce that $\ord_\tau(x)=\ell_\tau(x)+1$.

\medskip

We obtain  a version of the Knapp-Stein dimension theorem in our frameworks:
\begin{Corollary}\label{corSilberger_dimension_theorem}(see \cite{silberger1978knapp})
Let $\tau\in \UC_\C$. Then \[I_\tau(\tau)=\bigoplus_{w_R\in R_\tau} \C F_{w_R}.\vb_\tau.\] If moreover $\sigma_s=\sigma_s'$ for all $s\in \SCC$, then $\End_{\HC-\mathrm{mod}}(I_\tau)\simeq \C[R_\tau]$.
\end{Corollary}

\begin{proof}
This follows from \cite[Proposition 5.13 (2) and Proposition 5.27]{hebert2021decompositions}.
\end{proof}

Note that under the notation of \cite{hebert2021decompositions}, Theorem~\ref{thmWeight_space} implies that $I_\tau(\tau,\Wta)=\C \vb_\tau$. We can thus apply the results from 5.3 to 5.6 of \cite{hebert2021decompositions}, when $\tau\in \UC_\C$.  In particular, we have a description of the submodules and the irreducible quotients of $I_\tau$ when $\tau\in \UC_\C$, see \cite[Theorem 5.38]{hebert2021decompositions}. 

\medskip

We also obtain Kato's irreducibility criterion:
\begin{Corollary}\label{corKatos_irreducibility_criterion}(see \cite[Theorem 2.4]{kato1982irreducibility})

Let $\tau\in T_\C$. Then $I_\tau$ is irreducible if and only if:\begin{enumerate}
\item $\tau\in \UC_\C$,

\item $\Wta=\{1\}$.
\end{enumerate}

\end{Corollary}

\begin{proof}
By \cite[Proposition 4.17 and Theorem 4.8]{hebert2018principal}, if $I_\tau$ is irreducible, then $W_\tau=\Wta$ and $\tau\in \UC_\C$. Conversely, let $\tau\in \UC_\C$ be such that $\Wta=W_\tau$. Then $R_\tau=W_\tau/\Wta=\{1\}$ and by Corollary~\ref{corSilberger_dimension_theorem}, $I_\tau(\tau)=\C \vb_\tau$. By \cite[Theorem 4.8]{hebert2018principal} we deduce that $I_\tau$ is irreducible.
\end{proof}

\printindex

\bibliography{/home/auguste_pro/Documents/Projets/bibliographie.bib}
\bibliographystyle{plain}

\bigskip
\par\noindent Universit\'e de Lorraine, CNRS, Institut Élie Cartan de Lorraine, F-54000 Nancy, France, UMR 7502

E-mail: auguste.hebert@univ-lorraine.fr.

\end{document}